\documentclass[10pt,reqno]{amsart}
\usepackage{geometry}
\usepackage{amssymb}
\usepackage{amsmath}
\usepackage{changes}
\usepackage[mathscr]{euscript}
\usepackage[small]{caption}
\usepackage{mathtools}
\usepackage{graphicx}
\usepackage{xcolor}
\usepackage{comment}
\usepackage{changes}

\usepackage{setspace}
\usepackage{enumitem}

\usepackage{tikz}
\usetikzlibrary{shapes.geometric}

\usepackage[colorlinks=true, citecolor=blue, filecolor=black, linkcolor=black, urlcolor=black]{hyperref}

\allowdisplaybreaks

\makeatletter
\@addtoreset{equation}{section}
\makeatother

\theoremstyle{plain}
\newtheorem{theorem}{Theorem}[section]
\newtheorem{lemma}[theorem]{Lemma}
\newtheorem{proposition}[theorem]{Proposition}

\theoremstyle{remark}
\newtheorem{remark}[theorem]{Remark}

\newcounter{as}[section]

\newcommand{\abs}[1]{\left|#1\right|}

\newcommand{\RR}{\mathbb{R}}
\newcommand{\ZZ}{\mathbb{Z}}
\newcommand{\PP}{\mathbb{P}}
\newcommand{\EE}{\mathbb{E}}

\DeclareMathOperator{\erf}{erf}

\DeclareMathOperator{\CAP}{cap}
\DeclareMathOperator{\DIV}{div}

\definecolor{bblue}{rgb}{.2,0.2,.8}
\definechangesauthor[color=orange]{DM}

\begin{document}\raggedbottom

\title[Mean extinction time estimate for the contact process on stars]
{The Eyring--Kramers Law for the extinction time of \\ the contact process on stars}

\begin{abstract}
	In this paper, we derive a precise estimate for the mean extinction time 
	of the contact process with a fixed infection rate 
	on a star graph with $N$ leaves.
	Specifically, we determine not only
	the exponential main factor but also the exact sub-exponential
	prefactor in the asymptotic expression for the mean extinction time 
	as $N\to\infty$. 
	Previously, such detailed asymptotic information on the mean extinction time of the contact process was available exclusively for complete graphs.
	To obtain our results, we first
	establish an accurate estimate for the stationary distribution
	of a modified contact process, employing special function
	theory and refined Laplace's method. Subsequently, we apply a
	recently developed potential theoretic approach for analyzing metastability in
	non-reversible Markov processes, enabling us to deduce the asymptotic expression. 
	The integration of these methodologies constitutes a novel approach developed in this paper, one which has not been utilized previously in the study of the contact process.
\end{abstract}


\author{Younghun Jo}

\address{Y. Jo. Department of Mathematical Sciences, Seoul National University, Republic of Korea.}
\email{starrysky422@snu.ac.kr}

\maketitle

\section{Introduction}\label{section:intro}

The contact process is a class of interacting particle systems
	introduced by Harris \cite{harris} in 1974
	to model the spread of an infection within a population.
It is also referred to as the susceptible-infected-susceptible (SIS) model
	by mathematical epidemiologists.
In the contact process on a graph $G$,
	infected sites recover with rate $1$
	and transmit infection to each adjacent site with rate $\lambda$.
For a comprehensive introduction to the contact process,
	we refer the reader to \cite[Part I]{liggett-99}
	and the references therein.

A critical feature of the contact process on finite graphs is that
	the infection eventually becomes extinct almost surely.
This extinction occurs when all vertices are simultaneously healthy at some point in time,
	after which this all-healthy state persists indefinitely.
Consequently, on finite graphs,
	the analysis typically focuses on estimating the hitting time $\tau_{G}$ of this all-healthy configuration,
	termed the \emph{extinction time}.
The initial configuration usually considered is the all-infected state.
The infection rate $\lambda$ influences the extinction time;
	higher rates lead to prolonged durations before the process reaches the healthy state,
	due to more frequent transmissions of infection by each infected node.
Accordingly, characterizing the quantitative relationship between
	the extinction time $\tau_G$ and the rate $\lambda$ is the main agenda
	in the study of the contact process.

For large enough $\lambda$, 
	the infection may persist for an exceedingly long period.
To be more precise,
	consider a family of finite graphs $(G_N)_{N\ge 1}$ of increasing size
	and a fixed infection rate $\lambda > 0$.
Then, for sufficiently large $\lambda$,
	the extinction time grows exponentially:
	there exists $c_\lambda > 0$ such that
\begin{equation*}
	\PP[\tau_{G_N} \ge \exp(c_\lambda|G_N|)] \xrightarrow{N\to\infty} 1,
\end{equation*}
	where for any graph $G$ we denote its number of vertices by $|G|$.
This behavior has been demonstrated across various types of graphs,
	including boxes in $\ZZ^d$ \cite{cassandro, schonmann,
	durrett-88a, durrett-88b, durrett-89, mountford-93, mountford-99},
	general finite graphs \cite{mountford-16, schapira-17},
	and random graphs \cite{chatterjee, mountford-13, linker, schapira-21}.

The exponential growth of the extinction time can be interpreted as
	an instance of metastability, a widespread phenomenon characterized by
	prolonged persistence in transient states within stochastic systems.
Metastability often indicates that 
	the system undergoes a first-order phase transition,
	wherein the transition times between metastable states grow exponentially
	as $N \to \infty$, 
	where $N$ stands for the system size or spatial resolution.
This behavior is common in a wide class of models,
	including condensing interacting particle systems,
	low-temperature spin systems,
	and stochastic partial differential equations.
We refer the reader to the monographs \cite{bovier-15, olivieri}
	for a detailed discussion on recent developments on this topic.

The metastable dynamics of the contact process has been extensively analyzed
	across a broad class of graphs.
Mountford, Mourrat, Valesin, and Yao \cite{mountford-16} proved that
	if $\lambda > \lambda_c(\ZZ)$,
	there exists $c_\lambda > 0$ such that
\begin{equation*}
	\EE \tau_{G_N} \ge \exp(c_\lambda|G_N|)
\end{equation*}
	for connected graphs $G_N$ with uniformly bounded degree.
Here, $\lambda_c(\ZZ)$ denotes the critical infection rate \cite{liggett-99}
	associated with the phase transition of the contact process on $\ZZ$,
	defined as the infimum infection rate 
	at which an infection initiated from a single vertex
	survives indefinitely with positive probability.
Schapira and Valesin \cite{schapira-17} relaxed the bounded-degree constraint,
	establishing a slightly weaker result.

If we shift our focus to more concrete families, more is known.
A series of works \cite{cassandro, schonmann, durrett-88b, mountford-93, mountford-99}
	showed that
	if $\lambda$ is sufficiently large,
	then there exists a sharp exponent for the extinction time on boxes of $\ZZ^d$.
More precisely, there exists $c_\lambda > 0$ such that
\begin{equation}\label{eq:ldp}
	\frac{1}{|B_N|}\log \tau_{B_N} \xrightarrow{N\to\infty} c_\lambda\;
	\text{ in probability,}
\end{equation}
	where $B_N$ is a box of $\ZZ^d$ with side length $N$.
Schapira and Valesin \cite{schapira-21}
	proved an analogous result for a variety of random graph models.
It is worth pointing out that
	the condition of sufficiently large $\lambda$ is essential:
	for certain graphs, including boxes in $\ZZ^d$ \cite{durrett-88a},
	the extinction time grows only logarithmically 
	with respect to the number of vertices when $\lambda$ is sufficiently small.
We also mention that the logarithmic estimate \eqref{eq:ldp} is hard to obtain
	unless we are able to exploit a specific geometric features
	of the underlying graph.
For instance, when considering the periodic lattice $\ZZ_N^d$,
	rather than a lattice with open boundary conditions,
	a logarithmic estimate of the form \eqref{eq:ldp} has not been obtained.

Next, we shift our focus from logarithmic estimates 
	to precise asymptotics for the extinction time $\EE \tau_G$.
For processes exhibiting metastable behavior,
	a sharp asymptotic formula for the mean hitting time 
	from one metastable set to another
	is often referred to as the Eyring--Kramers law \cite{eyring, kramers}.
Obtaining the Eyring--Kramers law for the extinction time of the contact process is
	known to be highly challenging, with rigorous results available only for 
	the complete graph due to its simple geometric structure.
The strong spatial symmetry of complete graphs allows 
	the contact process to be reduced to 
	a one-dimensional nearest neighbor random walk,
	whose hitting times can be explicitly computed.
Even slight asymmetries in the underlying graph structure 
	(e.g., a one-dimensional cycle $\mathbb{Z}_N$)
	introduce significant complications.

Historically,
	studies of extinction times often employed coarse methodologies
	such as percolation theory and coupling methods.
Meanwhile, a significant advancement in metastability theory,
	specifically in establishing the Eyring--Kramers law, 
	was achieved in the influential works of Bovier, Eckhoff, Gayrard, and Klein 
	\cite{bovier-01,bovier-04}.
They developed a precise framework for quantifying
	key metastability metrics, such as transition times and hitting probabilities,
	in potential theoretic terms, including equilibrium potentials and capacities.
As a result, they developed a robust methodology
	for analyzing metastable behavior of reversible dynamics,
	now known as the potential theoretic approach.
This framework has recently been extended to nonreversible settings 
	in \cite{gaudilliere, landim, seo-19, slowik}.
We refer the reader to \cite{bovier-15, seo-arxiv} 
	for a detailed description on this scheme.

The main focus of this article is the contact process on star graphs.
Refer to Figure \ref{fig:star} for examples of configurations of the contact process on a star.
A star graph, characterized by a central node directly connected to all other nodes,
	exhibits one discernible asymmetry
	and serves as a natural model for analyzing epidemic hubs.
The contact process on stars was initially investigated by Pemantle \cite{pemantle-92}
	within his study on contact process on trees.
Durrett and Huang \cite{huang} recently provided upper and lower bounds
	on the exponent governing the extinction time.
More recently, Wang \cite{wang}
	observed from the perspective of large deviation theory that
	the explicit exponent of the mean extinction time must equal
\begin{equation*}
	c_\lambda
	= 2\log(1+\lambda) - \log(1+2\lambda)
\end{equation*}
	so that the mean extinction time scales as $e^{c_\lambda N}$,
	up to a subexponential prefactor, on a star with $N$ leaves.
The significance of stars in the analysis of the contact process 
	arises from their role as fundamental building blocks 
	or long-time infection reservoirs within larger graph structures.
This theme appears prominently in studies involving diverse underlying graph types,
	especially in random graph models,
	including preferential attachment models \cite{berger-05},
	power law random graphs \cite{chatterjee, mountford-13},
	Galton--Watson trees \cite{huang},
	and random hyperbolic graphs \cite{linker}.

In this study, we establish the Eyring--Kramers law for
	the extinction time of the contact process on star graphs
	by applying the potential theoretic approach for non-reversible systems.
This marks the first nontrivial instance of obtaining sharp estimates for the mean extinction time
	and the first application of potential theoretic principles to the study of the contact process.
Our main result is stated explicitly in Theorem \ref{thm:eyring kramers}.

Regarding our proof methodology, 
	one notable challenge arises from the requirement of irreducibility
	traditionally imposed by potential theory,
	a condition not satisfied by the contact process due to its absorbing states.
We overcome this issue by introducing modified processes and 
	quasi-stationary distributions,%
	\footnote{%
		In this paper, 
			the term ``quasi-stationary distribution'' is used 
			in a non-standard way. 
		Readers seeking clarification 
			may refer to Subsection \ref{subsection:quasi-stationary}.
	}
	thereby enabling the application of the potential theoretic framework 
	to systems with absorbing states.
We employ techniques from special function theory and refined Laplace's method
	to derive precise estimates for the quasi-stationary distribution.
Additionally, the inherent non-reversibility of the contact process poses
	significant technical obstacles,
	which we address by leveraging recent advances 
	in the analysis of non-reversible dynamics based on flow structures.

\section{Model and Main Results}\label{section:notation}

\begin{figure}
	\centering
	\begin{tikzpicture}
		\tikzstyle{vh}=[circle, draw, solid, fill=none, inner sep=0pt, minimum width=15pt]
		\tikzstyle{vi}=[circle, draw, solid, fill=gray, inner sep=0pt, minimum width=15pt]
		\node[vh] (o1) at (0,0) {$0$};
		\node[draw=none,minimum size=3cm,regular polygon,regular polygon sides=8] (a) at (o1) {};
		\foreach \x in {1,2,4,6,7}{
			\node[vi] (a\x) at (a.corner \x) {$1$};
		}
		\foreach \x in {3,5,8}{
			\node[vh] (a\x) at (a.corner \x) {$0$};
		}
		\foreach \x in {1,...,8}{
			\draw (o1) to (a\x);
		}
		\node[shift={(.15,.5)}] at (a8) {$A$};
		
		\node at (2.5,0) {$\Longrightarrow$};

		\node[vi] (o2) at (5,0) {$1$};
		\node[draw=none,minimum size=3cm,regular polygon,regular polygon sides=8] (b) at (o2) {};
		\foreach \x in {1,2,4,6,7}{
			\node[vi] (b\x) at (b.corner \x) {$1$};
		}
		\foreach \x in {3,5,8}{
			\node[vh] (b\x) at (b.corner \x) {$0$};
		}
		\foreach \x in {1,...,8}{
			\draw (o2) to (b\x);
		}
		\node[shift={(.15,.5)}] at (b8) {$A$};

		\node at (7.5,0) {$\Longrightarrow$};

		\node[vi] (o3) at (10,0) {$1$};
		\node[draw=none,minimum size=3cm,regular polygon,regular polygon sides=8] (c) at (o3) {};
		\foreach \x in {1,2,4,6,7,8}{
			\node[vi] (c\x) at (c.corner \x) {$1$};
		}
		\foreach \x in {3,5}{
			\node[vh] (c\x) at (c.corner \x) {$0$};
		}
		\foreach \x in {1,...,8}{
			\draw (o3) to (c\x);
		}
		\node[shift={(.15,.5)}] at (c8) {$A$};
	\end{tikzpicture}
	\caption{%
		Example configurations of the contact process on a star.
		Note that the number of infected leaves cannot increase when the hub is healthy (see the configuration on the left). 
		A healthy leaf (denoted by $A$) can become infected only after the hub has been reinfected.}
	\label{fig:star}
\end{figure}

Throughout this article, we fix a value $\lambda > 0$.
Let $G = (V,E)$ be an undirected simple graph with bounded degree,
	and write $x \sim y$ when vertices $x$ and $y$ are adjacent.
The contact process on $G$ with infection rate $\lambda$ is
	a continuous-time Markov process $(\eta_t)_{t\ge 0}$ 
	taking values in $\{0,1\}^V$.
By identifying $\eta_t$ with the subset $\{x\in V: \eta_t(x) = 1\}$ of $V$,
	the transition rates of the process are given by
\begin{equation*}
	\begin{cases} \smallskip
		\eta_t \to \eta_t \setminus \{x\} \text{ for each $x\in\eta_t$ with rate $1$,} \\  \smallskip
		\eta_t \to \eta_t \cup \{x\} \text{ for each $x\not\in\eta_t$ with rate $\lambda\cdot |\{y\in \eta_t: x\sim y\}|$},
	\end{cases}
\end{equation*}
where $|A|$ denotes the cardinality of a set $A$.
At time $t$, a vertex $x$ is said to be \emph{healthy} if $\eta_t(x) = 0$,
	and \emph{infected} if $\eta_t(x) = 1$.
Note that the \emph{all-healthy state} $\eta_t \equiv 0$ is an absorbing state of the process.

Let $S_N$ be the star graph with one hub and $N$ leaves.
Since the $N$ leaves are homogeneous,
	the contact process on the star can be faithfully described by
	a new Markov process $(o_t,n_t)_{t\ge 0}$ taking values in $\{0,1\}\times [0,N]$,
	where $o_t$ denotes the status of the hub---healthy or infected---%
	and $n_t$ is the number of infected leaves.%
	\footnote{%
		In this article,
		we let $[a,b]$ stand for the intersection of the closed interval $[a,b]$ with $\ZZ$, for $a,b\in\RR$.
	}
In essence, the process behaves as a continuous-time random walk on a ladder graph,
	whose transition rates defined as:
\begin{equation*}
	\begin{cases} \smallskip
		(1,n) \to (1,n+1) &\text{with rate $\lambda(N-n)$,} \\  \smallskip
		(1,n) \to (1,n-1) &\text{with rate $n$,} \\  \smallskip
		(1,n) \to (0,n) &\text{with rate $1$,} \\  \smallskip
		(0,n) \to (0,n-1) &\text{with rate $n$,} \\  \smallskip
		(0,n) \to (1,n) &\text{with rate $\lambda n$.}
	\end{cases}
\end{equation*}
The dynamics of this process depend significantly on the status of the hub.
Specifically, when the hub is healthy ($o_t = 0$),
	the number of infected leaves cannot increase until the hub becomes reinfected.
This dependency introduces a critical asymmetry, 
	affecting the overall behavior of infection spread within the graph. 
We refer to Figure \ref{fig:ladder} for an illustration.

Our main result establishes the Eyring--Kramers law
	for the extinction time of the contact process on stars.
\begin{theorem}[Eyring--Kramers law]\label{thm:eyring kramers}
	Let $\varepsilon > 0$ be given.
	Then, for each $x \in \{0,1\}\times [\varepsilon N, N]$, we have%
	\footnote{%
		Let $(f_N)_{N\ge 1} = (f_N(n))_{N\ge 1}$ and $(g_N)_{N\ge 1} = (g_N(n))_{N\ge 1}$ be
			collections of real functions in $n$.
		We write $f_N = O(g_N)$, $g_N = \Omega(f_N)$, or $f_N \lesssim g_N$ if
			there exists some constant $C > 0$ such that
		\begin{equation*}
			|f_N(n)| \le C |g_N(n)|
			\text{ for all $N \ge 1$ and $n$.}
		\end{equation*}
		We write $f_N = o(g_N)$ or $f_N \ll g_N$ if
		\begin{equation*}
			\lim_{N\to\infty} \sup_n\frac{f_N(n)}{g_N(n)} = 0.
		\end{equation*}
		In particular, $f_N = o(1)$ if $\sup_n f_N(n) \to 0$ as $N \to\infty$.
		We write $f_N \simeq g_N$ if $f_N = g_N(1+o(1))$.}
	\begin{equation}\label{eq:eyring kramers}
		\EE_{x}\tau_{(0,0)}
		= \kappa_\lambda
			N^{-\frac{1}{1+2\lambda}}
			\Bigl(\frac{(1+\lambda)^2}{1+2\lambda}\Bigr)^{N}
			(1+o(1))
	\end{equation}
	as $N \to \infty$, where the error term $o(1)$ is uniform in $x$ and the constant $\kappa_\lambda$ is explicitly given by 
	\[
		\kappa_\lambda
		= \Bigl(\frac{1+\lambda}{\lambda}\Bigr)^{\frac{2}{1+2\lambda}}
			\Gamma\Bigl(\tfrac{2(1+\lambda)}{1+2\lambda}\Bigr).
	\]
	Here, $\Gamma(a)$ denotes the gamma function.
	In particular,
	\begin{equation}\label{eq:large deviation}
		\lim_{N\to\infty}
		\sup_{x \in \{0,1\}\times [\varepsilon N, N]}
		\frac{1}{N}\log \EE_{x}\tau_{(0,0)}=
		  2\log(1+\lambda) - \log(1+2\lambda) .
	\end{equation}
\end{theorem}
The exponent of the mean extinction time given explicitly 
	on the right-hand side of \eqref{eq:large deviation}
	aligns with the observation previously made by Wang \cite{wang}.

The paper is organized as follows.
In Section \ref{section:energy landscape},
	we explore the behavior of the quasi-stationary distribution of the process,
	establishing a precise asymptotic formula as $N \to \infty$.
In Section \ref{section:theoretical background},
	we introduce a potential theoretic framework suitable for non-reversible dynamics.
Finally, in Section \ref{section:proof},
	we prove our main theorem by expressing the mean extinction time
	in terms of capacity and equilibrium potential,
	and by subsequently estimating the capacity through variational principles.

\section{Energy Landscape}\label{section:energy landscape}

We begin by examining the quasi-stationary distribution of the contact process on the star graph $S_N$
	with a fixed infection rate $\lambda > 0$.
In Subsection \ref{subsection:quasi-stationary},
	we derive an explicit representation of the quasi-stationary distribution 
	by examining the stationarity conditions of the process.
In Subsection \ref{subsection:asymptotic},
	we determine the sharp asymptotic behavior of the quasi-stationary distribution
	by employing techniques from special function theory and refined Laplace's method.
Some direct consequences of these asymptotics are discussed in Subsection \ref{subsection:landscape property}.

\subsection{Quasi-Stationary Distribution}\label{subsection:quasi-stationary}

\begin{figure}
	\centering
	\begin{tikzpicture}
		\def\l{4.5}
		\tikzstyle{vh}=[circle, draw, solid, fill=none, inner sep=0pt, minimum width=5pt]
		\tikzstyle{vi}=[circle, draw, solid, fill=gray, inner sep=0pt, minimum width=5pt]
		\tikzstyle{rect}=[draw, solid, minimum width=2cm, minimum height=2cm, node distance=2cm]
		
		\foreach \n/\m in {3/4,4/5,5/6}{
			\node[rect] (r0\n) at (\l*\n-4*\l,0) {};
			\node[vh] (o0\n) at (\l*\n-4*\l,0.3) {};
			\node[draw=none,minimum size=1cm,regular polygon,regular polygon sides=8] (a0\n) at (o0\n) {};
			\foreach \x in {1,...,\n}
				\node[vi] (a0\n\x) at (a0\n.corner \x) {};
			\foreach \x in {\m,...,8}
				\node[vh] (a0\n\x) at (a0\n.corner \x) {};
			\foreach \x in {1,...,8}
				\draw (o0\n) to (a0\n\x);
		}
		\foreach \n/\m in {3/4,4/5,5/6}{
			\node[rect] (r1\n) at (\l*\n-4*\l,-\l) {};
			\node[vi] (o1\n) at (\l*\n-4*\l,-\l+0.3) {};
			\node[draw=none,minimum size=1cm,regular polygon,regular polygon sides=8] (a1\n) at (o1\n) {};
			\foreach \x in {1,...,\n}
				\node[vi] (a1\n\x) at (a1\n.corner \x) {};
			\foreach \x in {\m,...,8}
				\node[vh] (a1\n\x) at (a1\n.corner \x) {};
			\foreach \x in {1,...,8}
				\draw (o1\n) to (a1\n\x);
		}
		\node at (-\l,-0.6) {$(0,n-1)$}; \node at (-\l,-\l-0.6) {$(1,n-1)$};
		\node at (0,-0.6) {$(0,n)$}; \node at (0,-\l-0.6) {$(1,n)$};
		\node at (\l,-0.6) {$(0,n+1)$}; \node at (\l,-\l-0.6) {$(1,n+1)$};

		\draw[->,thick] (r03) edge[bend right] node[midway, fill=white]{\footnotesize$\lambda (n-1)$} (r13);
		\draw[->,thick] (r13) edge[bend right] node[midway, fill=white]{\footnotesize$1$} (r03);

		\draw[->,thick] (r04) edge[bend right] node[midway, fill=white]{\footnotesize$\lambda n$} (r14);
		\draw[->,thick] (r14) edge[bend right] node[midway, fill=white]{\footnotesize$1$} (r04);

		\draw[->,thick] (r05) edge[bend right] node[midway, fill=white]{\footnotesize$\lambda (n+1)$} (r15);
		\draw[->,thick] (r15) edge[bend right] node[midway, fill=white]{\footnotesize$1$} (r05);

		\draw[->,thick] (r04) edge[bend right] node[midway, fill=white]{\footnotesize$n$} (r03);
		\draw[->,thick] (r05) edge[bend right] node[midway, fill=white]{\footnotesize$n+1$} (r04);
		\draw[->,thick] (r14) edge[bend right] node[midway, fill=white]{\footnotesize$n$} (r13);
		\draw[->,thick] (r15) edge[bend right] node[midway, fill=white]{\footnotesize$n+1$} (r14);

		\draw[->,thick] (r13) edge[bend right] node[midway, fill=white]{\footnotesize$\lambda(N-n+1)$} (r14);
		\draw[->,thick] (r14) edge[bend right] node[midway, fill=white]{\footnotesize$\lambda(N-n)$} (r15);

		\node at (-6.5,-2.25) {$\cdots$}; \node at (6.5,-2.25) {$\cdots$};

		\node[draw=gray, draw opacity=0.8, solid, rectangle, rounded corners, minimum width = \l*2 cm + 3cm, minimum height = 2.6cm, label={[anchor=center,fill=white,text=gray,text opacity=0.8]left:$u_n$}] (UN) at (0,0) {};

		\node[draw=gray, draw opacity=0.8, solid, rectangle, rounded corners, minimum width = \l*2 cm + 3cm, minimum height = 2.6cm, label={[anchor=center,fill=white,text=gray,text opacity=0.8]left:$v_n$}] (VN) at (0,-\l) {};
	\end{tikzpicture}
	\caption{%
		Transition rates for the contact process on a star.
		Dark circles denote infected vertices, and light circles denote healthy vertices.
	}
	\label{fig:ladder}
\end{figure}

Recall that the contact process is generally not irreducible,
	as it possesses a unique absorbing state---the all-healthy state.
Therefore, potential theory cannot be directly applied 
	because the stationary distribution is a Dirac mass at the absorbing state.
To address this, we slightly modify the original process by adding supplementary 
	transition rates from the absorbing state to other states, 
	thereby rendering the modified process irreducible.
Importantly, this type of modification does not affect the extinction time.

A natural choice for these supplementary transition rates,
	independent of the underlying graph structure, 
	involves setting them proportional to 
	the stationary measure conditioned on non-extinction---%
	often referred to as the quasi-stationary distribution of the process.
Under this choice, the stationary distribution of the resulting process would be 
	a convex combination of the quasi-stationary distribution and a Dirac mass 
	concentrated at the absorbing state.
However, deriving sharp asymptotic estimates for the quasi-stationary distribution 
	of an absorbing process generally poses a highly challenging problem. 

In view of this difficulty, we introduce only a single supplementary transition rate: 
\begin{equation*}
	(0,0) \to (1,0) \quad \text{with rate } \alpha,
\end{equation*}
where $\alpha > 0$.
We refer to this modified process as the \emph{regenerative process}.
Let $\nu = \nu_{N,\lambda,\alpha}$ denote 
	the stationary distribution of the regenerative process.
Due to the structural simplicity of the contact process on stars, 
	it turns out that
	$\nu$ can indeed be expressed as a convex combination of
	Dirac masses concentrated at the absorbing state $(0,0)$, 
	the states $(1,0)$ and $(0,1)$, 
	and the stationary distribution of the process 
	restricted to the set $(\{0,1\}\times [0,N])\setminus\{(0,0)\}$.
We note that the specific choice of $\alpha$ does not affect the subsequent analysis.

Although the stationary distribution $\nu$ of the regenerative process 
	is technically distinct from the quasi-stationary distribution, 
	it remains conceptually analogous,
	as it effectively assumes the role of 
	a stationary distribution for the absorbing process 
	within the potential theoretic analysis presented in subsequent sections.
For this reason, with a slight abuse of terminology, 
	we refer to $\nu$ as the quasi-stationary distribution 
	throughout this paper.

For computational convenience,
	we introduce a scaled measure $\mu = \mu_{N,\lambda,\alpha}$
	defined by $\nu = \frac{1}{Z_{N,\lambda}}\mu$, 
	where the scaling constant $Z_{N,\lambda}$ 
	is given by $Z_{N,\lambda} = \nu(1,N)^{-1}$.
By construction, we set the measure at the all-infected state to $\mu(1,N) = 1$.
We refer to $\mu$ as the quasi-stationary measure.
Furthermore, we introduce the notation:
\begin{equation*}
	u_n = \mu(0,n),\qquad v_n = \mu(1,n), \qquad 0\le n\le N,
\end{equation*}
	to represent the quasi-stationary measure of states
	in which the hub is healthy and infected, respectively.

By examining the stationarity conditions,
	we can readily derive the following $3$-term recurrence relations
	for the sequences $(u_n)_{0\le n\le N}$ and $(v_n)_{0\le n\le N}$.

\begin{proposition}[$3$-term recurrence relation for the quasi-stationary distribution]\label{prop:recurrence}
	Let the sequences $(u_n)_{0\le n\le N}$ and $(v_n)_{0\le n\le N}$ be as above.
	Then, for all $0\le n\le N$, it holds that
	\begin{align}
		&v_n = (1+\lambda) a_n - a_{n+1}, \label{eq:recurrence vn-1}\\
		&a_{n+1} = \lambda(N-n)v_n - (n+1)v_{n+1}, \label{eq:recurrence vn-2}\\
		&(n+1)a_{n+2} - (n+2+\lambda(N+1))a_{n+1} + \lambda(1+\lambda)(N-n)a_n = 0, \label{eq:recurrence vn-3}\\
		&(n+2)v_{n+2} - (n+2+\lambda N)v_{n+1} + \lambda(1+\lambda)(N-n)v_n = 0, \label{eq:recurrence vn-4}
	\end{align}
	where $u_{N+2} = u_{N+1} = v_{N+2} = v_{N+1} = 0$ and
	\[
		a_n = \begin{cases} \smallskip
			n u_n & \text{if $n\neq 0$,} \\ \smallskip
			\frac{\alpha}{1+\lambda}u_0 & \text{if $n = 0$.}
		\end{cases}
	\]
\end{proposition}

\begin{proof}
	The stationarity conditions at states with a healthy hub 
		yield \eqref{eq:recurrence vn-1}.
	Additionally, the conditions at states with an infected hub give the equations
	\begin{equation*}
		(n+1+\lambda(N-n))v_n
		= \lambda a_n + \lambda (N-n+1)v_{n-1} + (n+1)v_{n+1}
	\end{equation*}
	for all $1\le n\le N$.
	Subtracting the above equation with $n+1$ in place of $n$ 
		from the original equation multiplied by $1+\lambda$,
		we obtain 
	\begin{align*}
		&(n+2)v_{n+2} - (n+2+\lambda N)v_{n+1} + \lambda (1+\lambda)(N-n)v_n \\
		&= (n+1)v_{n+1} - (n+1+\lambda N)v_{n} + \lambda (1+\lambda)(N-n+1)v_{n-1}. 
	\end{align*}
	This identity shows that 
		the left-hand side of equation \eqref{eq:recurrence vn-4} 
		is constant. 
	Moreover, combining the stationarity conditions at the states 
		$(1,N)$ and $(0,N)$ shows that this constant must indeed be zero, 
		thus proving \eqref{eq:recurrence vn-4}.
	The equation \eqref{eq:recurrence vn-3} follows by a similar argument, 
		while the equation \eqref{eq:recurrence vn-2} can be proved using 
		backward induction on $n$.
\end{proof}

The stationary measure for a modified contact process on stars has
	previously been considered by Cator and Mieghem \cite{cator}.
Their setting involves a different version of modification:
	instead of adding supplementary rates,
	they removed all transitions leading to the all-healthy state
	and considered the trace process (cf. Subsection \ref{subsection:trace process}) restricted to the states with an infected hub.
The stationary distribution of their modified process, denoted by $\pi_n = \pi(1,n)$, 
	is essentially a restriction of the stationary distribution of the contact process 
	to the set $(\{0,1\}\times [0,N])\setminus\{(0,0)\}$.
We also note that 
	Bhamidi, Nam, Nguyen, and Sly \cite{bhamidi-21} 
	considered yet another type of modified contact process on finite trees, 
	designed to enable a recursive argument over the tree height 
	when analyzing the extinction time through the stationary distribution.

According to \cite[Equation (12)]{cator}, 
	the recurrence relation \eqref{eq:recurrence vn-4} derived above 
	for the sequence $(v_n)_{0\le n\le N}$
	also applies to the distribution $(\pi_n)_{0\le n\le N}$, 
	except at the point $n = 0$.
Consequently, the sequence $(v_n)_{0\le n\le N}$ is a constant multiple of 
	$(\pi_n)_{0\le n\le N}$, only except at $n = 0$.
In the same paper, Cator and Mieghem also explicitly solved the recurrence relation 
	to obtain an exact form for the quasi-stationary distribution.
For completeness,
	we briefly reformulate their computations and results below.

\begin{proposition}
	\label{prop:recurrence solution}
	Let $(\pi_n)_{0\le n\le N}$ be the sequence satisfying
	$\pi_1 = \lambda N \pi_0$, $\sum_{n=0}^N \pi_n = 1$, and
	\begin{equation*}
		(n+2)\pi_{n+2} - (n+2 + \lambda N)\pi_{n+1} + \lambda(1+\lambda)(N-n) \pi_n = 0
	\end{equation*}
	for $1\le n\le N-1$ where $\pi_{N+1} = 0$.
	Then, we have%
	\footnote{We note that the solution presented here 
			slightly differs from that given in the original paper.
		By carefully examining the computations line-by-line,
			one can verify that our solution is indeed consistent with the stated recurrence relation.}
	\begin{equation}\label{eq:recurrence solution}
	\begin{split}
		\pi_n
		=& - \frac{b}{\lambda(1+2\lambda)}
		\biggl[\sum_{j=n}^{N-1} (-1)^{j-n} \binom{N-1}{j}\binom{j}{n}
			B\Bigl(\tfrac{1}{1+2\lambda},j+1\Bigr)
			\Bigl(\frac{\lambda}{1+2\lambda}\Bigr)^j\biggr]
			(1+\lambda)^n \\
		&- \frac{b}{1+2\lambda}
		\biggl[\sum_{j=n-1}^{N-1} (-1)^{j-n+1} \binom{N-1}{j}\binom{j}{n-1}
			B\Bigl(\tfrac{1}{1+2\lambda},j+1\Bigr)
			\Bigl(\frac{\lambda}{1+2\lambda}\Bigr)^j\biggr]
			(1+\lambda)^{n-1} \\
		&- \frac{c - \frac{b}{\lambda}}{1+2\lambda}
			\biggl[\sum_{j=n}^{N} (-1)^{j-n} \binom{N}{j}\binom{j}{n}
				B\Bigl(\tfrac{1}{1+2\lambda},j+1\Bigr)
				\Bigl(\frac{\lambda}{1+2\lambda}\Bigr)^j\biggr]
				(1+\lambda)^{n}
	\end{split}
	\end{equation}
	for $1\le n\le N$,
	and
	\begin{equation}\label{eq:recurrence solution at 0}
	\begin{split}
		\pi_0
		=& - \frac{b}{\lambda(1+2\lambda)}
		\sum_{j=0}^{N-1} (-1)^{j} \binom{N-1}{j}
			B\Bigl(\tfrac{1}{1+2\lambda},j+1\Bigr)
			\Bigl(\frac{\lambda}{1+2\lambda}\Bigr)^j \\
		&- \frac{c - \frac{b}{\lambda}}{1+2\lambda}
			\sum_{j=0}^{N} (-1)^{j} \binom{N}{j}
				B\Bigl(\tfrac{1}{1+2\lambda},j+1\Bigr)
				\Bigl(\frac{\lambda}{1+2\lambda}\Bigr)^j,
	\end{split}
	\end{equation}
	where $B(a,b)$ is a beta function. 
	Here, $b = 2\pi_2 - \lambda N(\lambda N + 1 - \lambda)\pi_0$ and $c = -\pi_0$
	and they also satisfy the equation 
	\begin{equation}\label{eq:recurrence solution constant relation}
	\begin{split}
		1
		={}& \frac{c}{1+2\lambda}
			\sum_{j=0}^N \binom{N}{j}
			B\Bigl(\tfrac{1}{1+2\lambda},j+1\Bigr)
			\Bigl(\frac{\lambda^2}{1+2\lambda}\Bigr)^j \\
		&+ \frac{(1+\lambda)b}{\lambda(1+2\lambda)}
			\sum_{j=0}^{N-1} \binom{N-1}{j}
			B\Bigl(\tfrac{1}{1+2\lambda},j+1\Bigr)
			\Bigl(\frac{\lambda^2}{1+2\lambda}\Bigr)^j \\
		&- \frac{b}{\lambda(1+2\lambda)}
			\sum_{j=0}^{N} \binom{N}{j}
			B\Bigl(\tfrac{1}{1+2\lambda},j+1\Bigr)
			\Bigl(\frac{\lambda^2}{1+2\lambda}\Bigr)^j.
	\end{split}
	\end{equation}
\end{proposition}

\begin{remark}
	The first summation in the expression \eqref{eq:recurrence solution} can be transformed as follows.
	\begin{align*}
		&\quad \sum_{j=n}^{N-1} (-1)^{j-n} \binom{N-1}{j}\binom{j}{n}
			B\Bigl(\tfrac{1}{1+2\lambda},j+1\Bigr)
			\Bigl(\frac{\lambda}{1+2\lambda}\Bigr)^j \\
		&= \binom{N-1}{n}\Bigl(\frac{\lambda}{1+2\lambda}\Bigr)^{n}
		\int_0^1 \sum_{l=0}^{N-n-1} (-1)^l \binom{N-n-1}{l}
			\Bigl(\frac{\lambda}{1+2\lambda}\Bigr)^l
			t^{l+n}(1-t)^{-\frac{2\lambda}{1+2\lambda}} \,dt \\
		&= \binom{N-1}{n}\Bigl(\frac{\lambda}{1+2\lambda}\Bigr)^{n}
		\int_0^1 t^{n}\Bigl(1 - \frac{\lambda}{1+2\lambda}t\Bigr)^{N-n-1}
			(1-t)^{-\frac{2\lambda}{1+2\lambda}} \,dt \\
		&= \frac{1+2\lambda}{\lambda} \binom{N-1}{n}
			\int_0^{\frac{\lambda}{1+2\lambda}} u^{n}(1-u)^{N-n-1}
				\Bigl(1 - \frac{1+2\lambda}{\lambda}u\Bigr)^{-\frac{2\lambda}{1+2\lambda}} \,du.
	\end{align*}
	By performing a similar transformation for the other two lines, we obtain an alternative integral expression for the solution:
	\begin{equation}\label{eq:recurrence solution integral}
	\begin{split}
		\pi_n
		=& -\frac{b}{\lambda^2}\binom{N-1}{n} (1+\lambda)^n
			\int_0^{\frac{\lambda}{1+2\lambda}}u^n(1-u)^{N-n-1}
				\Bigl(1 - \frac{1+2\lambda}{\lambda}u\Bigr)^{-\frac{2\lambda}{1+2\lambda}} \,du \\
		& -\frac{b}{\lambda}\binom{N-1}{n-1} (1+\lambda)^{n-1}
		\int_0^{\frac{\lambda}{1+2\lambda}}u^{n-1}(1-u)^{N-n-1}
			\Bigl(1 - \frac{1+2\lambda}{\lambda}u\Bigr)^{-\frac{2\lambda}{1+2\lambda}} \,du \\
		& -\frac{c - \frac{b}{\lambda}}{\lambda}\binom{N}{n} (1+\lambda)^n
			\int_0^{\frac{\lambda}{1+2\lambda}}u^n(1-u)^{N-n}
				\Bigl(1 - \frac{1+2\lambda}{\lambda}u\Bigr)^{-\frac{2\lambda}{1+2\lambda}} \,du.
	\end{split}
	\end{equation}
	Similarly, from equation \eqref{eq:recurrence solution at 0}, we obtain
	\begin{equation}\label{eq:recurrence solution at 0 integral}
	\begin{split}
		\pi_0
		=& - \frac{b}{\lambda(1+2\lambda)}
			\int_0^1
				\Bigl(1 - \frac{\lambda}{1+2\lambda}t\Bigr)^{N-1}
				(1-t)^{-\frac{2\lambda}{1+2\lambda}} \,dt \\
		& - \frac{c - \frac{b}{\lambda}}{1+2\lambda}
			\int_0^1
				\Bigl(1 - \frac{\lambda}{1+2\lambda}t\Bigr)^{N}	
				(1-t)^{-\frac{2\lambda}{1+2\lambda}} \,dt.
	\end{split}
	\end{equation}
\end{remark}

\begin{proof}[Proof of Proposition \ref{prop:recurrence solution}]
	Define a generating function
	\[
		G(z) = \sum_{n=0}^N \pi_n z^n
	\]
	so that $G(1) = 1$. The recurrence relation for the sequence $(\pi_n)_{0\le n\le N}$ yields the differential equation
	\[
		(1-z-\lambda(1+\lambda)z^2)G'(z)
		+(\lambda(1+\lambda)Nz - (1+\lambda N))G(z)
		= bz + c,
	\]
	where $b = 2\pi_2 - \lambda N(\lambda N + 1 - \lambda)\pi_0$ and $c = -\pi_0$.

	The homogeneous solution of the equation
	\[
		(1-z-\lambda(1+\lambda)z^2)G_h'(z)
		+(\lambda(1+\lambda)Nz - (1+\lambda N))G_h(z)
		= 0,
	\]
	which can be found by separating variables, is
	\[
		G_h(z) = C(z-r_1)^{c_1}(z-r_2)^{c_2},
	\]
	where $C$ is a constant and
	\[
		\begin{cases} \smallskip
			r_1 = \frac{1}{1+\lambda}, \\ \smallskip
			r_2 = - \frac{1}{\lambda},
		\end{cases}
		\qquad
		\begin{cases} \smallskip
			c_1 = -\frac{1}{1+2\lambda}, \\ \smallskip
			c_2 = N + \frac{1}{1+2\lambda}.
		\end{cases}
	\]
	
	Now we variate the constant so that we express $G(z) = C(z)G_h(z)$.
	Note that
	\[
		C(1) = \Bigl(\frac{\lambda}{1+\lambda}\Bigr)^{N + \frac{2}{1+2\lambda}}.
	\]
	Putting into the primary differential equation gives
	\[
		C'(z)
		= -\frac{bz + c}{\lambda(1+\lambda)(z-r_1)^{c_1+1}(z-r_2)^{c_2+1}}.
	\]
	We may integrate both sides from $1$ to $z$ to obtain
	\[
		C(z)
		= \Bigl(\frac{\lambda}{1+\lambda}\Bigr)^{N + \frac{2}{1+2\lambda}}
		- \frac{1}{\lambda(1+\lambda)}
		\int_1^z \frac{bu + c}{(u-r_1)^{c_1+1}(u-r_2)^{c_2+1}} \,du.
	\]
	By changing variables with $t = \frac{1+2\lambda}{1+\lambda}\frac{1}{1 + \lambda u}$,
		we obtain
	\begin{equation}\label{eq:variate constant}
	\begin{split}
		C(z)
		={}& \Bigl(\frac{\lambda}{1+\lambda}\Bigr)^{N + \frac{2}{1+2\lambda}} \\
		&- \frac{b}{\lambda(1+\lambda)}
			\Bigl(\frac{\lambda(1+\lambda)}{1+2\lambda}\Bigr)^N
			\Bigl[B\Bigl(\tfrac{1+2\lambda}{(1+\lambda)(1+\lambda z)} ;
				\tfrac{1}{1+2\lambda}, N\Bigr)
			- B\Bigl(\tfrac{1+2\lambda}{(1+\lambda)^2} ;
				\tfrac{1}{1+2\lambda}, N\Bigr)\Bigr] \\
		&- \frac{c - \frac{b}{\lambda}}{1+2\lambda}
			\Bigl(\frac{\lambda(1+\lambda)}{1+2\lambda}\Bigr)^N
			\Bigl[B\Bigl(\tfrac{1+2\lambda}{(1+\lambda)(1+\lambda z)} ;
				\tfrac{1}{1+2\lambda}, N+1\Bigr)
			- B\Bigl(\tfrac{1+2\lambda}{(1+\lambda)^2} ;
				\tfrac{1}{1+2\lambda}, N+1\Bigr)\Bigr],
	\end{split}
	\end{equation}
	where $B(x;a,b)$ denotes the incomplete beta function.
	We note that for a nonnegative integer $k$, the identity
	\[
		B(k,b) - B(x;k,b)
		= \sum_{j=0}^{k-1} \binom{k-1}{j} B(j+1,b) x^{k-1-j} (1-x)^{j+b}
	\]
	holds, and this can be readily shown by induction on $k$.
	After applying this identity to \eqref{eq:variate constant},
		we obtain the solution \eqref{eq:recurrence solution} by expanding.
	The relation \eqref{eq:recurrence solution constant relation} follows by
		checking analyticity of $G(z)$.
\end{proof}

\subsection{Asymptotic Analysis}\label{subsection:asymptotic}

\filbreak
\begin{theorem}[Asymptotic behavior of the quasi-stationary measure, high infection regime]\label{thm:stationary asymptotic high}
	Let
	\begin{equation*}
		v_n^\mathrm{high} = \binom{N}{n} \lambda^{n-N} f_{N,\lambda}(n),
	\end{equation*}
	where
	\begin{equation*}
		f_{N,\lambda}(n)
		= \prod_{k=n}^{N-1}
			\Bigl[1 + \frac{1}{(1+2\lambda)(k+1)-\lambda N}\Bigr].
	\end{equation*}
	Then, the scaled quasi-stationary measure $v_n = \mu(1,n)$ at states with an infected hub
		satisfies the following estimates.
	
	\begin{enumerate}
		\item Let $\varepsilon > 0$ be given, and let $\delta = \varepsilon N$.
		Then, for $n \in [\frac{\lambda}{1+2\lambda}N + \delta,N]$, we have
		\begin{equation*}
			v_n = v_n^\mathrm{high} (1 + O(N^{-1}))
		\end{equation*}
		as $N\to\infty$, where the error term is uniform in $n$.

		\item Let $0 < a < \frac{1}{2}$ be given, and let $\delta = N^{\frac{1}{2} + a}$.
		Then, for $n \in [\frac{\lambda}{1+2\lambda}N + \delta,N]$, we have
		\begin{equation*}
			v_n = v_n^\mathrm{high} (1 + O(N^{-2a}))
		\end{equation*}
		as $N\to\infty$, where the error term is uniform in $n$.
	\end{enumerate}
\end{theorem}

\begin{theorem}[Asymptotic behavior of the quasi-stationary measure, low infection regime]\label{thm:stationary asymptotic low}
	Let
	\begin{equation*}
		v_n^\mathrm{low} = C_{N,\lambda} (1+\lambda)^n g_{N,\lambda}(n),
	\end{equation*}
	where
	\begin{equation*}
		g_{N,\lambda}(n)
		= \Bigl(
			1 - \frac{1+2\lambda}{\lambda}\frac{n}{N}
		\Bigr)^{-\frac{2\lambda}{1+2\lambda}}
	\end{equation*}
	and
	\begin{equation*}
		C_{N,\lambda}^{-1}
		= \frac{\lambda}{1+2\lambda}
			N B\Bigl(\tfrac{1}{1+2\lambda}, N\Bigr)
			\Bigl(\frac{\lambda(1+\lambda)}{1+2\lambda}\Bigr)^N.
	\end{equation*}
	Then, the scaled quasi-stationary measure $v_n = \mu(1,n)$ at states with an infected hub
		satisfies the following estimates.
	
	\begin{enumerate}
		\item Let $\varepsilon > 0$ be given, and let $\delta = \varepsilon N$.
		Then, for $n \in [0, \frac{\lambda}{1+2\lambda}N-\delta]$, we have
		\begin{equation*}
			v_n = v_n^\mathrm{low} (1 + O(N^{-1}))
		\end{equation*}
		as $N\to\infty$, where the error term is uniform in $n$.

		\item Let $0 < a < \frac{1}{2}$ be given, and let $\delta = N^{\frac{1}{2} + a}$.
		Then, for $n \in [0, \frac{\lambda}{1+2\lambda}N - \delta]$, we have
		\begin{equation*}
			v_n = v_n^\mathrm{low} (1 + O(N^{-2a}))
		\end{equation*}
		as $N\to\infty$, where the error term is uniform in $n$.
	
		\item Let $0 < a < \frac{1}{2}$ be given, and let $\delta = N^{\frac{1}{2} + a}$.
		Then, for $n \in [\frac{\lambda}{1+2\lambda}N-\delta,\frac{\lambda}{1+2\lambda}N + \delta]$, we have
		\begin{equation*}
			C_{N,\lambda} (1+\lambda)^n e^{-2\delta}
			\ll v_n
			\lesssim C_{N,\lambda} (1+\lambda)^n
				\max\Bigl\{
					N^{\frac{\lambda}{1+2\lambda}},
					N^\frac{1}{2}
					\Bigl|
						\frac{\lambda}{1+2\lambda}-\frac{n}{N}
					\Bigr|^{\frac{1}{1+2\lambda}}
				\Bigr\}
		\end{equation*}
		as $N\to\infty$, where the error term is uniform in $n$.
	\end{enumerate}
\end{theorem}

\begin{remark}
	By applying Theorem \ref{thm:stationary asymptotic low} 
		to the recurrence relation \eqref{eq:recurrence vn-1}, 
		we obtain 
	\begin{equation}\label{eq:u0 asymptotic}
		\alpha \mu(0,0) = \lambda N C_{N,\lambda} (1 + O(N^{-1})).
	\end{equation}
\end{remark}

\begin{remark}
	The theorem highlights a significant difference in the asymptotic behavior of $v_n$
		between two distinct regimes:
		one characterized by a large number of infected leaves,
		and the other by a smaller number.
	The prefactors $f_{N,\lambda}(n)$ and $g_{N,\lambda}(n)$ exhibit subexponential growth with respect to $N$.
	Consequently, in the high infection regime,
		the dominant exponential factor of $v_n$ is $\binom{N}{n}\lambda^{n-N}$,
		whereas in the low infection regime,
		it is dominated by $C_{N,\lambda} (1+\lambda)^n$.

	To understand the rationale behind these exponential dominance factors of $v_n$,
		consider the dual modes of the dynamics, contingent on the status of the hub.
	When $n$ is sufficiently large,
		the infection rate of the hub, proportional to $n$,
		significantly exceeds its constant cure rate of $1$.
	This allows the approximation that the hub remains perpetually infected.
	Under this assumption, the leaves evolve independently,
		each following a Markov chain
		with a rate of infection of $\lambda$ and a recovery rate of $1$.
	Consequently, the distribution of $n$ converges rapidly to
		a binomial distribution $\operatorname{Binom}(N,\frac{\lambda}{1+\lambda})$,
		which explains the presence of the factor $\binom{N}{n}\lambda^{n-N}$.

	Conversely, in scenarios where $n$ is small,
		the mode characterized by a healthy hub 
		predominantly influences the energy landscape 
		despite its brief temporal occurrence.
	In such cases, when the hub is healthy,
		the subsequent transition in the process is either the reinfection of the hub
		with probability $\frac{\lambda}{1+\lambda}$
		or the recovery of an infected leaf with probability $\frac{1}{1+\lambda}$.
	Thus, once the hub is cured,
		the number of subsequently cured leaves before the hub is reinfected
		follows a (truncated) geometric distribution
		with success probability $\frac{\lambda}{1+\lambda}$,
		which accounts for the factor $(1+\lambda)^n$.

	The transition between these regimes occurs at a point where
		the influences of both modes are comparably significant.
	A pathwise approach, grounded in large deviation theory,
		indicates that this transitional state corresponds to when
		the proportion of infected leaves is approximately $\frac{\lambda}{1+2\lambda}$.
\end{remark}

\begin{remark}
	The sequence $(v_n)_{0\le n\le N}$ attains
		its global maximum at $n \simeq \frac{\lambda}{1+\lambda}N$
		and its local minima at $n = 0$ and $n = N$.
	The former corresponds to the metastable state of the process,
		defined as a state exhibiting maximal stability,
		while the point $n = 0$ represents the saddle point of the process.
	Hence, the exact exponent of the potential barrier is given by 
	\[
		\lim_{N\to\infty}\frac{1}{N}\log(v_{\lfloor\frac{\lambda}{1+\lambda}N\rfloor} / v_0)
		= 2\log(1+\lambda) - \log(1+2\lambda).
	\]
	This suggests the large deviation principle
	\[
		\lim_{N\to\infty}\frac{1}{N}\log\mathbb{E}\tau
		= 2\log(1+\lambda) - \log(1+2\lambda),
	\]
	which was previously observed by Wang \cite{wang}.
\end{remark}

To outline the proof of the asymptotic behavior of $v_n$,
	we first address the high infection regime.
Our primary objective is to demonstrate that
	the deviation between $v_n$ and $v_n^{\mathrm{high}}$ remains controlled
	as $n$ decreases from $N$.
This result is established inductively by leveraging the recurrence relation
	outlined in \eqref{eq:recurrence vn-4}.
Turning to the low infection regime,
	we initiate our analysis by deriving an explicit representation for $C_{N,\lambda}$.
Subsequently, we partition the low infection regime into three distinct segments.
The first segment considers values of $n$ that are
	neither close to $0$ nor near the transitional point $\frac{\lambda}{1+2\lambda}N$.
For these, we apply Laplace's method in a more refined manner to 
	the integral representation for $v_n$ in \eqref{eq:recurrence solution integral}.
The second segment addresses values of $n$ close to $0$,
	where our approach parallels that employed in the high infection regime,
	albeit initiated from the outcomes obtained previously by Laplace's method.
The final segment deals with values of $n$ near the transitional point $\frac{\lambda}{1+2\lambda}N$.
Here, we once again utilize the integral expression in \eqref{eq:recurrence solution integral}
	to derive suitable asymptotic bounds for $v_n$.

\begin{proof}[Proof of Theorem \ref{thm:stationary asymptotic high}]
	Define $x_n = v_n / v_n^{\mathrm{high}}$ and $y_n = x_n/x_{n+1}$.
	Then we have $x_N = x_{N-1} = 1$ and $y_{N-1} = 1$.
	Furthermore, the sequence $(y_n)_{0\le n\le N-1}$ satisfies a recurrence relation
	\begin{equation}\label{eq:yn recurrence high}
	\begin{split}
		&\quad (1+\lambda)(n+1)
			\Bigl(1 + \frac{1}{(1+2\lambda)(n+2)-\lambda N}\Bigr)
			\Bigl(1 + \frac{1}{(1+2\lambda)(n+1)-\lambda N}\Bigr)(y_n-1) \\
		&= \frac{2(1+\lambda)((1+\lambda)(n+1)-\lambda N)}{((1+2\lambda)(n+2)-\lambda N)((1+2\lambda)(n+1)-\lambda N)}
		+ \lambda (N-n-1) \frac{y_{n+1} - 1}{y_{n+1}},
	\end{split}
	\end{equation}
	which can be obtained directly from \eqref{eq:recurrence vn-4}.
	
	We now prove the theorem in two steps.

	\smallskip

	\noindent\textbf{Step 1.}
	$n \in [\frac{\lambda}{1+2\lambda}N+\varepsilon N, N]$, where $\varepsilon > 0$.
	
	\nopagebreak
	By relation \eqref{eq:yn recurrence high}, it holds that 
	\begin{equation}\label{eq:high asymptotic proof main step}
	\begin{split}
		\abs{y_n-1}
		\le{}& \frac{2\abs{(1+\lambda)(n+1)-\lambda N}}{(n+1)((1+2\lambda)(n+2)-\lambda N)((1+2\lambda)(n+1)-\lambda N)} \\
			& + \frac{\lambda (N-n-1)}{(1+\lambda)(n+1)} \frac{\abs{y_{n+1} - 1}}{y_{n+1}} \\
		\le{}& \frac{A}{N^2}
		+ (1 - r) \frac{|y_{n+1} - 1|}{y_{n+1}}
	\end{split}
	\end{equation}
	for constants $A>0$ and $0 < r < 1$ that depend only on $\lambda$ and $\varepsilon$.
	Choose a constant $B>0$, depending only on $\lambda$ and $\varepsilon$, 
		so that the inequality 
	\[
		\frac{A}{N^2} + (1-r)\frac{B/N^2}{1 - B/N^2} \le \frac{B}{N^2}
	\]
	holds for all sufficiently large $N$. For instance, $B = \frac{2A}{r}$ works. Since
	\[
		\abs{y_{N-1}-1} = 0 \le \frac{B}{N^2},
	\]
	we inductively deduce that $\abs{y_n - 1} \le \frac{B}{N^2}$ for all $n\in [\frac{\lambda}{1+2\lambda}N+\varepsilon N, N]$. Thus, we obtain
	\[
		\abs{\log x_n}
		\le \sum_{k=n}^{N-1} \abs{\log y_k}
		\le N \Bigl|\log\Bigl(1 - \frac{B}{N^2}\Bigr)\Bigr|
		= O(N^{-1}),
	\]
	implying that $\abs{x_n - 1} = O(N^{-1})$ for all $n\in [\frac{\lambda}{1+2\lambda}N+\varepsilon N, N]$.

	\smallskip
	\noindent\textbf{Step 2.} 
	$n \in [\frac{\lambda}{1+2\lambda}N+\delta, N]$, 
		where $\delta = N^{\frac{1}{2}+a}$, $0 < a < \frac{1}{2}$.
	
	\nopagebreak
	We proceed analogously to the previous step.
	Write $n = \frac{\lambda}{1+2\lambda}N + d_n$.
	We claim that
	\begin{equation*}
		\abs{y_n - 1} = O(Nd_n^{-3}).
	\end{equation*}
	Similarly to \eqref{eq:high asymptotic proof main step},
		we have
	\begin{equation*}
	\begin{split}
		\abs{y_n-1}
		\le{}& \frac{2\abs{(1+\lambda)(n+1)-\lambda N}}{(n+1)((1+2\lambda)(n+2)-\lambda N)((1+2\lambda)(n+1)-\lambda N)} \\
			& + \frac{\lambda (N-n-1)}{(1+\lambda)(n+1)} \frac{\abs{y_{n+1} - 1}}{y_{n+1}} \\
		\le{}& \frac{A}{d_n^2}
		+ \Bigl(1 - r \frac{d_n}{N}\Bigr) \frac{\abs{y_{n+1} - 1}}{y_{n+1}}
	\end{split}
	\end{equation*}
	for constants $A>0$ and $0 < r < 1$ that depend only on $\lambda$ and $\varepsilon$.
	Put $B = \frac{2A}{r}$, then it holds that
	\[
		\frac{A}{d_n^2}
			+ \Bigl(1 - r \frac{d_n}{N}\Bigr) \frac{BN(d+1)^{-3}}{1 - BN(d+1)^{-3}}
		\le \frac{BN}{d_n^3}
	\]
	for all sufficiently large $N$.
	Hence, we inductively deduce that $|y_n - 1| \le BNd_n^{-3}$
		for all $n \in [\frac{\lambda}{1+2\lambda}N+\delta, N]$.
	Thus, we obtain
	\begin{align*}
		\abs{\log x_n}
		&\le \sum_{k=n}^{N-1} \abs{\log y_k}
		\le \sum_{k=n}^{N-1}\Bigl|\log\Bigl(1 - B\frac{N}{(k - \frac{\lambda}{1+2\lambda}N)^3}\Bigr)\Bigr|
		\lesssim \frac{1}{N} \int_{\frac{\delta}{N}}^{\frac{1+\lambda}{1+2\lambda}}\frac{dt}{t^3}
		= O(N^{-2a}),
	\end{align*}
	and the proof is complete.
\end{proof}

\begin{lemma}
	\label{lem:integration Taylor}
	Let $-1 < a < 0$ and $0 < p < 1$ be given. Then, for each $m \in \mathbb{Z}_{> 0}$, we have 
	\begin{equation*}
		\int_0^1 (1-pt)^N (1-t)^a \,dt
		= \sum_{k=0}^{m-1} (-1)^k\binom{a}{k} p^{-k-1} B(k+1,N+1) + O(N^{-m})
	\end{equation*}
	as $N\to\infty$,
	where $B(a,b)$ denotes the beta function.
\end{lemma}

\begin{proof}
	Denote the integral on the left-hand side by $I_N$.
	The term $(1-pt)^N$ decays exponentially as $N \to \infty$ 
		for $t\in [0,1]$ bounded away from $0$,
		while the term $(1-t)^a$ has a unique algebraic singularity at $1$.
	Hence, we may replace the integration interval with $[0,1/2]$ without affecting the asymptotic behavior,
		so that the range does not contain $1$.

	Now, by Taylor's theorem, we have 
	\begin{align*}
		I_N
		&= \int_0^{1/2} (1-pt)^N (1-t)^a \,dt + O(N^{-m}) \\
		&= \sum_{k=0}^{m-1} (-1)^k\binom{a}{k}\int_0^{1/2} (1-pt)^N t^k \,dt
			+ \int_0^{1/2} (1-pt)^N O(t^{m}) \,dt
			+ O(N^{-m}).
	\end{align*}
	Note that
	\begin{align*}
		\int_0^{1/2} t^k (1-pt)^N \,dt
		&= p^{-k-1} B(p/2; k+1,N+1),
	\end{align*}
	where $B(x;a,b)$ denotes the incomplete beta function.
	By the asymptotic expansion \cite[Equation (8.18.1)]{nist} of incomplete beta functions,
		the error between $B(p/2;k+1,N+1)$ and $B(k+1,N+1)$ is exponentially small in $N$.
	Thus,
	\begin{align*}
		I_N
		&= \sum_{k=0}^{m-1} (-1)^k\binom{a}{k}\int_0^{1/2} (1-pt)^N t^k \,dt
			+ \int_0^{1/2} (1-pt)^N O(t^{m}) \,dt
			+ O(N^{-m}) \\
		&= \sum_{k=0}^{m-1} (-1)^k\binom{a}{k} p^{-k-1} B(k+1,N+1)
			+ O(N^{-m})
	\end{align*}
	since $B(m,N+1) = O(N^{-m})$.
\end{proof}

\begin{proposition}[Initial value condition]\label{prop:vn initial value}
	As $N\to\infty$, we have 
	\begin{equation*}
		\frac{v_1}{v_N} = C_{N,\lambda} (1+\lambda) (1+O(N^{-1})),
	\end{equation*}
	where $C_{N,\lambda}$ is as defined in Theorem \ref{thm:stationary asymptotic low}.
\end{proposition}

\begin{proof}
	We claim that
	\begin{equation}\label{eq:b over pi0}
		\frac{b}{\pi_0}
		= -\lambda^2 N^2 + \lambda(1+3\lambda) N
		+ O(1),
	\end{equation}
	where $b$ and $\pi_0$ are as defined in Proposition \ref{prop:recurrence solution}.

	Define 
	\[
		I_N
		= \int_0^1
			\Bigl(1 - \frac{\lambda}{1+2\lambda}t\Bigr)^{N}
			(1-t)^{-\frac{2\lambda}{1+2\lambda}} \,dt.
	\]
	Then from the integral expression \eqref{eq:recurrence solution at 0 integral},
		we have the relation
	\[
		\frac{b}{\pi_0}
		= \frac{\lambda(1+2\lambda - I_N)}{I_N - I_{N-1}}.
	\]
	We deduce the asymptotic expansion \eqref{eq:b over pi0}
		by applying Lemma \ref{lem:integration Taylor} and subsequently performing a long division.

	By putting $n=N$ in the solution \eqref{eq:recurrence solution}, we have
	\begin{align*}
		\frac{\pi_N}{\pi_0}
		=&-\frac{b}{(1+2\lambda)\pi_0}
			B\Bigl(\tfrac{1}{1+2\lambda},N\Bigr)
			\Bigl(\frac{\lambda}{1+2\lambda}\Bigr)^{N-1} (1+\lambda)^{N-1} \\
		&- \frac{c-\frac{b}{\lambda}}{(1+2\lambda)\pi_0}
			B\Bigl(\tfrac{1}{1+2\lambda},N+1\Bigr)
			\Bigl(\frac{\lambda}{1+2\lambda}\Bigr)^{N} (1+\lambda)^{N} \\
		= {}& \frac{1}{\lambda(1+2\lambda)}
			B\Bigl(\tfrac{1}{1+2\lambda},N+1\Bigr)
			\Bigl(\frac{\lambda(1+\lambda)}{1+2\lambda}\Bigr)^N
			\Bigl[\lambda + \Bigl(1
				- \frac{N+\frac{1}{1+2\lambda}}{N}\frac{1+2\lambda}{1+\lambda}\Bigr)
				\frac{b}{\pi_0}\Bigr] \\
		={}& \frac{\lambda}{1+\lambda} N \frac{1}{C_{N,\lambda}} (1+O(N^{-1})).
	\end{align*}
	Therefore, we conclude that
	\[
		\frac{v_1}{v_N}
		= \frac{\pi_1}{\pi_N}
		= \frac{\lambda N \pi_0}{\pi_N}
		= C_{N,\lambda} (1+\lambda)(1+O(N^{-1})),
	\]
	and this is precisely the assertion of the proposition.
\end{proof}

\begin{proof}[Proof of Theorem \ref{thm:stationary asymptotic low}]
	The proof is divided into three steps. 
	
	\smallskip
	
	\noindent\textbf{Step 1.} 
	$n\in [\varepsilon N, \frac{\lambda}{1+2\lambda}N-\delta]$, 
		where $\varepsilon > 0$, $0 < a < \frac{1}{2}$, 
			and $\delta = \varepsilon N$ or $\delta = N^{\frac{1}{2}+a}$.

	\nopagebreak
	We abbreviate $n/N = \beta$ and $\min\{\beta,\frac{\lambda}{1+2\lambda}-\beta\} = \gamma$.
	Recall from \eqref{eq:recurrence solution integral}:
	\begin{align*}
		\pi_n
		=& -\frac{b}{\lambda^2}\binom{N-1}{n} (1+\lambda)^n
			\int_0^{\frac{\lambda}{1+2\lambda}}u^n(1-u)^{N-n-1}
				\Bigl(1 - \frac{1+2\lambda}{\lambda}u\Bigr)^{-\frac{2\lambda}{1+2\lambda}} \,du \\
		& -\frac{b}{\lambda}\binom{N-1}{n-1} (1+\lambda)^{n-1}
		\int_0^{\frac{\lambda}{1+2\lambda}}u^{n-1}(1-u)^{N-n-1}
			\Bigl(1 - \frac{1+2\lambda}{\lambda}u\Bigr)^{-\frac{2\lambda}{1+2\lambda}} \,du \\
		& -\frac{c - \frac{b}{\lambda}}{\lambda}\binom{N}{n} (1+\lambda)^n
			\int_0^{\frac{\lambda}{1+2\lambda}}u^n(1-u)^{N-n}
				\Bigl(1 - \frac{1+2\lambda}{\lambda}u\Bigr)^{-\frac{2\lambda}{1+2\lambda}} \,du \\
		=& \biggl[
				\frac{-b}{\lambda}\beta^n(1-\beta)^{N-n}\binom{N}{n}
			\biggr]
			(1+\lambda)^{n-1}
			\Bigl(
				1 - \frac{1+2\lambda}{\lambda}\beta
			\Bigr)^{-\frac{2\lambda}{1+2\lambda}}
			\int_0^{\frac{\lambda}{1+2\lambda}} e^{-N P_n(t)}Q_n(t) \,dt,
	\end{align*}
	where
	\begin{align*}
		P_n(t)
		&= - \beta \log \frac{t}{\beta} - (1 - \beta)\log\frac{1-t}{1-\beta}, \\
		Q_n(t)
		&= \biggl(
				\frac{
					\frac{\lambda}{1+2\lambda} - t
				}{
					\frac{\lambda}{1+2\lambda} - \beta
				}
			\biggr)^{-\frac{2\lambda}{1+2\lambda}}
			\Bigl[
				\frac{1+\lambda}{\lambda}\frac{1-\beta}{1-t}
				+ \frac{\beta}{t}
				- \frac{1+\lambda}{\lambda}
			\Bigr].
	\end{align*}
	For simplicity of computation, we have omitted the term with $c$ from $Q_n(t)$.
	This omission can be justified by separately performing a similar computation as below for the term containing $c$, while noting from \eqref{eq:b over pi0} that $c/b = O(N^{-2})$.

	By Stirling's series \cite[Equation (5.11.3)]{nist},
	we have 
	\[
		\binom{N}{n} = \frac{1}{\sqrt{2\pi N} \beta^{n+1/2} (1-\beta)^{N-n+1/2}}
		\Bigl[
			1 + \frac{1}{12}\Bigl(1 - \frac{1}{\beta} - \frac{1}{1-\beta}\Bigr)\frac{1}{N} + O(N^{-2})
		\Bigr].
	\]
	Hence,
	\begin{equation*}
	\begin{split}
		\pi_n
		={}& \frac{-b}{\lambda}\frac{1}{\sqrt{2\pi N \beta (1-\beta)}}
			\Bigl[
				1 + \frac{1}{12}\Bigl(1 - \frac{1}{\beta} - \frac{1}{1-\beta}\Bigr)\frac{1}{N} + O(N^{-2})
			\Bigr] \\
		& \times
			(1+\lambda)^{n-1}
			\Bigl(
				1 - \frac{1+2\lambda}{\lambda}\beta
			\Bigr)^{-\frac{2\lambda}{1+2\lambda}}
			\int_0^{\frac{\lambda}{1+2\lambda}} e^{-N P_n(t)}Q_n(t) \,dt.
	\end{split}
	\end{equation*}

	The function $P_n(t)$ attains a unique minimum value $0$ at $t = \beta$ on the interval $(0,1)$, and we have
	\[
	\begin{split}
		P_n'(t)
			&= - \frac{\beta}{t} + \frac{1-\beta}{1-t}, \\
		P_n^{(3)}(t)
			&= -2\Bigl[ \frac{\beta}{t^3} - \frac{1-\beta}{(1-t)^3} \Bigr],
	\end{split}
	\qquad \qquad
	\begin{split}
		P_n''(t)
			&= \frac{\beta}{t^2} + \frac{1-\beta}{(1-t)^2}, \\
			P_n^{(4)}(t)
			&= 6\Bigl[ \frac{\beta}{t^4} + \frac{1-\beta}{(1-t)^4}\Bigr],
	\end{split}
	\]
	and
	\[
		Q_n'(t)
			= Q_n(t) A_n(t), \qquad \qquad
		Q_n''(t)
			= Q_n(t)[ A_n(t)^2 + A_n'(t)],
	\]
	where
	\[
		A_n(t)
		= \frac{2\lambda}{1+2\lambda}\frac{1}{\frac{\lambda}{1+2\lambda} - t}
			+ \frac{
				\frac{1+\lambda}{\lambda}\frac{1-\beta}{(1-t)^2}
				- \frac{\beta}{t^2}
			}{
				\frac{1+\lambda}{\lambda}\frac{1-\beta}{1-t}
				+ \frac{\beta}{t}
				- \frac{1+\lambda}{\lambda}
			}.
	\]
	We have assumed that $\gamma = \Omega(N^{-\frac{1}{2}+a})$, 
		so outside the interval $(\beta - \gamma/3,\beta+\gamma/3)$, 
		the exponential term in the integral decays rapidly as $N \to \infty$, 
		while $Q_n(t)$ has only algebraic singularities of order less than $1$.
	Hence, we may replace the integration interval with $(\beta-\gamma_1,\beta+\gamma_2)$ 
		without affecting the asymptotic behavior, where:
	\begin{itemize}
		\item $\gamma/3 < \gamma_i < \gamma/2$,
		\item $P_n(\beta - \gamma_1) = P_n(\beta + \gamma_2) = \kappa$.
	\end{itemize}
	Note that $P_n(t)$ and $Q_n(t)$ are smooth on the interval $(\beta-\gamma_1,\beta+\gamma_2)$.
	
	Now we apply Laplace's method to approximate the integrals.
	The theoretical background can be found in various textbooks, for instance, \cite[Section 3.7--9]{olver-74}.
	Define $p_{n,s}(t) = \frac{1}{(s+2)!}P_n^{(s+2)}(t)$ and $q_{n,s}(t) = \frac{1}{s!}Q_n^{(s)}(t)$ so that we have the Taylor expansions
	\begin{align*}
		P_n(t)
		&= p_0(t-\beta)^2 + p_1(t-\beta)^3 + p_2(t-\beta)^4 + \dots, \\
		Q_n(t)
		&= q_0 + q_1(t-\beta) + q_2(t-\beta)^2 + \dots,
	\end{align*}
	where $p_{n,s}(\beta) = p_s$ and $q_{n,s}(\beta) = q_s$.
	Substituting $v = P_n(t)$ for $t \ge \beta$ sufficiently close to $\beta$, we obtain 
	\[
		\int_\beta^{\beta+\gamma_2} e^{-NP_n(t)}Q_n(t) \,dt
		= \int_0^\kappa e^{-Nv} f(v) \,dv,
	\]
	where
	\[
		f(v)
		= \frac{Q_n(t)}{P_n'(t)}
		= a_{n,0}(\beta) v^{-1/2} + a_{n,1}(\beta) + a_{n,2}(\beta) v^{1/2} + a_{n,3}(\beta) v + \dots.
	\]
	Here, each coefficient $a_{n,s}$ is a polynomial expression in the terms $p_{n,i}$ and $q_{n,i}$, divided by an appropriate power of $p_{n,0}^{1/2}$. These coefficients can be explicitly computed through series reversion. The first few terms are as follows.
	\begin{equation}\label{eq:def of an}
		\begin{split}
			a_{n,0}
				&= \frac{q_{n,0}}{2p_{n,0}^{1/2}}, \qquad
			a_{n,1}
				= \frac{1}{2p_{n,0}^2}[p_{n,0}q_{n,1} - p_{n,1}q_{n,0}], \\
			a_{n,2}
				&= \frac{1}{16p_{n,0}^{7/2}}
				[
					8 p_{n,0}^2 q_{n,2}
					- 12 p_{n,0} p_{n,1} q_{n,1}
					+ 3(5p_{n,1}^2 - 4p_{n,0} p_{n,2}) q_{n,0}
				].
		\end{split}
	\end{equation}
	Similarly as before, we write $a_{n,s}(\beta) = a_s$ for brevity.
	We also have
	\[
		\int_{\beta-\gamma_1}^{\beta} e^{-NP_n(t)}Q_n(t) \,dt
		= \int_0^\kappa e^{-Nv} \widetilde{f}(v) \,dv,
	\]
	where
	\[
		\widetilde{f}(v) = a_0 v^{-1/2} - a_1 + a_2 v^{1/2} - a_3 v + \dots.
	\]
	Thus, the integral to estimate becomes
	\begin{equation*}
		\int_{\beta-\gamma_1}^{\beta+\gamma_2} e^{-NP_n(t)}Q_n(t) \,dt
		= 2 \int_0^\kappa e^{-Nv} F_n(v) \,dv,
	\end{equation*}
	where
	\begin{equation*}
		F_n(v) = a_0 v^{-1/2} + a_2 v^{1/2} + a_4 v^{3/2} + \dots.
	\end{equation*}
	
	For a fixed integer $s\ge 1$, the above integral can be expressed as
	\begin{equation}\label{eq:laplace expansion}
	\begin{split}
		&\quad \int_0^\kappa e^{-Nv} F_n(v) \,dv \\
		&= \int_0^\infty e^{-Nv} \sum_{i=0}^{s-1} a_{2i} v^{i-1/2} \,dv
			- \int_\kappa^\infty e^{-Nv} \sum_{i=0}^{s-1} a_{2i} v^{i-1/2} \,dv
			+ \int_0^\kappa e^{-Nv} R_{n,s}(v) \,dv \\
		&= \sum_{i=0}^{s-1} \Gamma\Bigl(i+\frac{1}{2}\Bigr)\frac{a_{2i}}{N^{i+1/2}}
			- \sum_{i=0}^{s-1} \Gamma\Bigl(i+\frac{1}{2}, N\kappa\Bigr)\frac{a_{2i}}{N^{i+1/2}}
			+ \int_0^\kappa e^{-Nv} R_{n,s}(v) \,dv,
	\end{split}
	\end{equation}
	where $\Gamma(a,x)$ is the incomplete gamma function and
	\begin{equation*}
		R_{n,s}(v)
		= F_n(v) - \sum_{i=0}^{s-1} a_{2i} v^{i-1/2}.
	\end{equation*}
	Note that the incomplete Gamma function can be bounded by
	\[
		\Gamma(r,x) \le \frac{e^{-x} x^r}{x - \max\{r-1,0\}}
		\qquad (x > \max\{r-1,0\}).
	\]
	See \cite[Equation (8.10.1)]{nist} and \cite[Chapter 3, Equation (2.14)]{olver-74}.
	Since we have $N\kappa = \Omega(N^{a})$ and $a_{2i}$ grows algebraically as $N \to \infty$,
		we may neglect the second term in the last line of \eqref{eq:laplace expansion}
		without affecting the asymptotic behavior.
	
	Now, we will show that the third term is also sufficiently small for some $s$.
	By Taylor's theorem, there exists some $t_*$ lying between $\beta$ and $t$ such that
	\begin{equation*}
		R_{n,s}(v) = a_{n,2s}(t_*) v^{s-1/2}.
	\end{equation*}

	First, we consider those values of $n$ satisfying $\gamma \ge \varepsilon > 0$.
	For $t\in [\beta - \gamma/2, \beta + \gamma/2]$, each of the terms $p_{n,i}$ and $q_{n,i}$ is uniformly bounded in $n$, and $p_{n,0}$ remains uniformly bounded away from $0$ in $n$. Consequently, each coefficient $a_{n,s}$ is also uniformly bounded in $n$, so we have 
	\[
		\biggl|\int_0^\kappa e^{-Nv} R_{n,s}(v) \,dv\biggr|
		\le \int_0^\kappa e^{-Nv} |a_{n,2s}(t_*)| v^{s-1/2} \,dv
		= O\Bigl(\frac{1}{N^{s+1/2}}\Bigr).
	\]
	This yields that
	\begin{align*}
		\int_0^\kappa e^{-Nv} F_n(v) \,dv
		&= \sum_{i=0}^{s-1} \Gamma\Bigl(i+\frac{1}{2}\Bigr)\frac{a_{2i}}{N^{i+1/2}}
			+ O\Bigl(\frac{1}{N^{s+1/2}}\Bigr),
	\end{align*}
	so
	\begin{align*}
		\pi_n
		={}& \frac{-b}{\lambda}\frac{1}{\sqrt{2\pi N \beta (1-\beta)}}
			\Bigl[
				1 + \frac{1}{12}\Bigl(1 - \frac{1}{\beta} - \frac{1}{1-\beta}\Bigr)\frac{1}{N} + O(N^{-2})
			\Bigr] \\
		& \times
			(1+\lambda)^{n-1}
			\Bigl(
				1 - \frac{1+2\lambda}{\lambda}\beta
			\Bigr)^{-\frac{2\lambda}{1+2\lambda}}
			\cdot 2 \int_0^\kappa e^{-Nv} F_n(v) \,dv \\
		={}& \frac{-b}{\lambda}\frac{1}{\sqrt{2\pi N \beta (1-\beta)}}
			\Bigl[
				1 + \frac{1}{12}\Bigl(1 - \frac{1}{\beta} - \frac{1}{1-\beta}\Bigr)\frac{1}{N} + O(N^{-2})
			\Bigr] \\
		& \times
			(1+\lambda)^{n-1}
			\Bigl(
				1 - \frac{1+2\lambda}{\lambda}\beta
			\Bigr)^{-\frac{2\lambda}{1+2\lambda}}
			\biggl[
				\sum_{i=0}^{s-1} 2\Gamma\Bigl(i+\frac{1}{2}\Bigr)\frac{a_{2i}}{N^{i+1/2}}
				+ O\Bigl(\frac{1}{N^{s+1/2}}\Bigr)
			\biggr].
	\end{align*}
	Putting $s = 2$ gives
	\begin{align*}
		\pi_n
		={}& \frac{-b}{\lambda}\frac{1}{\sqrt{2\pi N \beta (1-\beta)}}
			\Bigl[
				1 + \frac{1}{12}\Bigl(1 - \frac{1}{\beta} - \frac{1}{1-\beta}\Bigr)\frac{1}{N} + O(N^{-2})
			\Bigr] \\
		& \times
			(1+\lambda)^{n-1}
			\Bigl(
				1 - \frac{1+2\lambda}{\lambda}\beta
			\Bigr)^{-\frac{2\lambda}{1+2\lambda}}
			\Bigl[
				2\sqrt{\pi}\frac{a_0}{N^{1/2}}
				+ \sqrt{\pi}\frac{a_2}{N^{3/2}}
				+ O\Bigl(\frac{1}{N^{5/2}}\Bigr)
			\Bigr].
	\end{align*}
	Since $p_0 = \frac{1}{2\beta(1-\beta)}$ and $q_0 = 1$, we have
	\[
		a_0
		= \frac{q_0}{2p_0^{1/2}}
		= \sqrt{\frac{\beta(1-\beta)}{2}}.
	\]
	Recall from \eqref{eq:b over pi0}:
	\[
		\frac{b}{\pi_1}
		= -\lambda N + (1 + 3\lambda) + O(N^{-1}),
	\]
	so we obtain
	\begin{equation}\label{eq:low n precise}
	\begin{split}
		\pi_n
		={}& \frac{-b}{\lambda N}
			\Bigl[
				1 + \frac{1}{12}\Bigl(1 - \frac{1}{\beta} - \frac{1}{1-\beta}\Bigr)\frac{1}{N} + O(N^{-2})
			\Bigr] \\
		& \times
			(1+\lambda)^{n-1}
			\Bigl(
				1 - \frac{1+2\lambda}{\lambda}\beta
			\Bigr)^{-\frac{2\lambda}{1+2\lambda}}
			\Bigl[
				1
				+ \frac{a_2}{\sqrt{2\beta(1-\beta)}}\frac{1}{N}
				+ O\Bigl(\frac{1}{N^2}\Bigr)
			\Bigr] \\
		={}& \pi_1
			(1+\lambda)^{n-1}
			\Bigl(
				1 - \frac{1+2\lambda}{\lambda}\beta
			\Bigr)^{-\frac{2\lambda}{1+2\lambda}} \\
		& \times
			\Bigl[
				1
				+ \Bigl[
					- \frac{1+3\lambda}{\lambda}
					+ \frac{1}{12}\Bigl(1 - \frac{1}{\beta} - \frac{1}{1-\beta}\Bigr)
					+ \frac{a_2}{\sqrt{2\beta(1-\beta)}}
				\Bigr]\frac{1}{N}
				+ O\Bigl(\frac{1}{N^2}\Bigr)
			\Bigr].
	\end{split}
	\end{equation}

	Next, we deal with the case when $\gamma \le \varepsilon$ and $\gamma = \Omega(N^{-\frac{1}{2}+a})$.
	For $t \in [\beta - \gamma/2, \beta + \gamma/2]$, each of the terms $p_{n,i}$ is uniformly bounded, and
	\begin{align*}
		q_{n,0}(t)
			&= O(1), \\
		q_{n,1}(t)
			&= O(\gamma^{-1})
			  = O(N^{\frac{1}{2}(1-2a)}), \\
		q_{n,2}(t)
			&= O(\gamma^{-2})
			 = O(N^{1-2a}),
	\end{align*}
	so
	\begin{align*}
		a_{n,2}(t)
		&= \frac{1}{16p_{n,0}^{7/2}}
			[
				8 p_{n,0}^2 q_{n,2}
				- 12 p_{n,0} p_{n,1} q_{n,1}
				+ 3(5p_{n,1}^2 - 4p_{n,0} p_{n,2}) q_{n,0}
			]
		= O(N^{1-2a}).
	\end{align*}
	Now we put $s = 1$ and proceed similarly as before. Since we have
	\[
		\biggl|\int_0^\kappa e^{-Nv} R_{n,1}(v) \,dv\biggr|
		\le \int_0^\kappa e^{-Nv} |a_{n,2}(t_*)| v^{1/2} \,dv
		= O(N^{- \frac{1}{2} - 2a}),
	\]
	the integral to estimate has an asymptotic expression
	\begin{align*}
		\int_0^\kappa e^{-Nv} F_n(v) \,dv
		&= \sqrt{\pi}\frac{a_0}{N^{1/2}}
			+ O(N^{- \frac{1}{2} - 2a}).
	\end{align*}
	Hence,
	\begin{align*}
		\pi_n
		={}& \frac{-b}{\lambda}\frac{1}{\sqrt{2\pi N \beta (1-\beta)}}
			\Bigl[
				1 + \frac{1}{12}\Bigl(1 - \frac{1}{\beta} - \frac{1}{1-\beta}\Bigr)\frac{1}{N} + O(N^{-2})
			\Bigr] \\
		& \times
			(1+\lambda)^{n-1}
			\Bigl(
				1 - \frac{1+2\lambda}{\lambda}\beta
			\Bigr)^{-\frac{2\lambda}{1+2\lambda}}
			\Bigl[
				2\sqrt{\pi}\frac{a_0}{N^{1/2}}
				+ O(N^{- \frac{1}{2} - 2a})
			\Bigr] \\
		={}& \pi_1
			(1+\lambda)^{n-1}
			\Bigl(
				1 - \frac{1+2\lambda}{\lambda}\beta
			\Bigr)^{-\frac{2\lambda}{1+2\lambda}}
			\Bigl[
				1
				+ O(N^{-2a})
			\Bigr].
	\end{align*}
	
	\smallskip

	\noindent\textbf{Step 2}.
	$n \in [0, \varepsilon N]$, where $0 < \varepsilon < \frac{\lambda}{1+2\lambda}$.

	\nopagebreak
	In this regime, we may replace $g_{N,\lambda}(n)$ with
	\[
		\widetilde{g}_{N,\lambda}(n)
		= \prod_{k=0}^{n-1}\Bigl[
			1 + \frac{2\lambda}{\lambda N - (1+2\lambda)(k+1) - 1}
		\Bigr],
	\]
	since
	\begin{equation}\label{eq:alternative gn proof}
	\begin{split}
		\widetilde{g}_{N,\lambda}(n)
		&= \exp\biggl[
				\sum_{k=0}^{n-1}\log\Bigl(
					1 + \frac{2\lambda}{\lambda N - (1+2\lambda)(k+1) - 1}
				\Bigr)
			\biggr] \\
		&= \exp\biggl[
				\sum_{k=0}^{n-1} \Bigl(
					\frac{2\lambda}{\lambda N - (1+2\lambda)(k+1) - 1} + O(N^{-2})
				\Bigr)
			\biggr] \\
		&= \Bigl(
				1 - \frac{1+2\lambda}{\lambda}\frac{n}{N}
			\Bigr)^{-\frac{2\lambda}{1+2\lambda}}
			(1 + O(nN^{-2})) \\
		&= g_{N,\lambda}(n) (1 + O(N^{-1}))
	\end{split}
	\end{equation}
	by Taylor's theorem.
	We also replace $v_n^{\mathrm{low}}$ with $\widetilde{v}_n^{\mathrm{low}}$ accordingly.

	Define $x_n = v_n/\widetilde{v}_n^{\mathrm{low}}$
		and $y_n = x_{n}/x_{n+1}$.
	Set $n_0 = \lfloor\varepsilon N\rfloor$.
	We will show that $x_{n_0} = 1 + O(N^{-1})$ and $y_{n_0} = 1 + O(N^{-2})$.
	Recall from \eqref{eq:low n precise} in the previous step that,
		for each integer $n$ sufficiently close to $n_0$, we have
	\begin{equation}\label{eq:low n precise revisited}
		\frac{v_n}{v_n^{\mathrm{low}}}
		= \frac{v_1}{C_{N,\lambda}} \times
			\Bigl[
				1
				+ \Bigl[
					- \frac{1+3\lambda}{\lambda}
					+ \frac{1}{12}\Bigl(1 - \frac{1}{\beta} - \frac{1}{1-\beta}\Bigr)
					+ \frac{a_{n,2}(\beta)}{\sqrt{2\beta(1-\beta)}}
				\Bigr]\frac{1}{N}
				+ O\Bigl(\frac{1}{N^2}\Bigr)
			\Bigr],
	\end{equation}
	where $\beta = n/N$ and $a_{n,2}$ is as defined in \eqref{eq:def of an}.
	For these values of $n$, we have
	\begin{align*}
		\frac{g_{N,\lambda}(n+1)}{g_{N,\lambda}(n)}
		&= \Bigl(
				1 - \frac{1+2\lambda}{\lambda N - (1+2\lambda)n}
			\Bigr)^{-\frac{2\lambda}{1+2\lambda}}
		= \frac{\widetilde{g}_{N,\lambda}(n+1)}{\widetilde{g}_{N,\lambda}(n)}
			(1 + O(N^{-2})),
	\end{align*}
	so we may replace the left-hand side of \eqref{eq:low n precise revisited} with $x_n$.
	Define 
	\begin{equation*}
		D(n)
		= - \frac{1+3\lambda}{\lambda}
		+ \frac{1}{12}\Bigl(1 - \frac{1}{\beta} - \frac{1}{1-\beta}\Bigr)
		+ \frac{a_{n,2}(\beta)}{\sqrt{2\beta(1-\beta)}}.
	\end{equation*}
	Then we have
	\[
		x_n
		= \frac{v_1}{C_{N,\lambda}} \times
			\Bigl[
				1
				+ D(n)\frac{1}{N}
				+ O\Bigl(\frac{1}{N^2}\Bigr)
			\Bigr],
	\]
	so
	\begin{align*}
		y_n
		= \frac{x_n}{x_{n+1}}
		= 1 + (D(n) - D(n+1))\frac{1}{N} + O(N^{-2}).
	\end{align*}
	Hence, it suffices to show that
	\[
		D(n_0 + 1) - D(n_0) = O(N^{-1}).
	\]
	This follows immediately from the observation that
		$D(n)$ is differentiable with respect to $\beta$,
		and its derivative is uniformly bounded in a neighborhood of $\beta = n_0/N$.

	The sequence $(y_n)_{0\le n\le N-1}$ satisfies the recurrence relation
	\begin{align*}
		&\quad (1+\lambda)(n+2)
			\Bigl(1 + \frac{2\lambda}{\lambda N - (1+2\lambda)(n+1) - 1}\Bigr)
			\Bigl(1 + \frac{2\lambda}{\lambda N - (1+2\lambda)n - 1}\Bigr) 
			\frac{1 - y_{n+1}}{y_{n+1}} \\
		&= - \frac{2\lambda^2(1+4\lambda)(N-n)}{(\lambda N - (1+2\lambda)(n+1) - 1)(\lambda N - (1+2\lambda)n - 1)}
		+ \lambda (N-n) (1 - y_n),
	\end{align*}
	which can be derived directly from \eqref{eq:recurrence vn-4}.
	Consequently, we have 
	\begin{align*}
		\abs{y_{n}-1}
		\le{}& \frac{2\lambda (1+4\lambda)}{(\lambda N - (1+2\lambda)(n+1) - 1)(\lambda N - (1+2\lambda)n - 1)} \\
			&+ \frac{(1+\lambda)(n+2)}{\lambda(N-n)}
				\Bigl(1 + \frac{2\lambda}{\lambda N - (1+2\lambda)(n+1) - 1}\Bigr)
				\Bigl(1 + \frac{2\lambda}{\lambda N - (1+2\lambda)n - 1}\Bigr) \\
			&\quad\times \frac{\abs{y_{n+1}-1}}{y_{n+1}} \\
		\le{}& \frac{A}{N^2}
		+ (1 - r) \frac{\abs{y_{n+1} - 1}}{y_{n+1}}
	\end{align*}
	with constants $A>0$ and $0 < r < 1$ that depend only on $\lambda$ and $\varepsilon$. Choose a constant $B>0$, depending only on $\lambda$ and $\varepsilon$, so that the inequality 
	\[
		\frac{A}{N^2} + (1-r)\frac{B/N^2}{1 - B/N^2} \le \frac{B}{N^2}
	\]
	holds for all sufficiently large $N$. For instance, $B = \frac{2A}{r}$ works. Since
	\[
		\abs{y_{n_0}-1} \le \frac{B}{N^2},
	\]
	by selecting a sufficiently large $A$, we inductively deduce that $\abs{y_n - 1} \le \frac{B}{N^2}$ for all $n\in [1,\varepsilon N]$. Thus, we obtain
	\[
		\abs{\log x_n}
		\le \sum_{k=n}^{n_0-1} \abs{\log y_k} + \abs{\log x_{n_0}}
		\le N\Bigl|\log\Bigl(1 - \frac{B}{N^2}\Bigr)\Bigr| + \abs{\log x_{n_0}}
		= O(N^{-1}),
	\]
	implying that $\abs{x_n - 1} = O(N^{-1})$ for all $n\in [1,\varepsilon N]$.

	\smallskip

	\noindent\textbf{Step 3.}
	$n \in [\frac{\lambda}{1+2\lambda}N - \delta, \frac{\lambda}{1+2\lambda}N + \delta]$,
		where $\delta = N^{\frac{1}{2}+a}$, $0 < a < \frac{1}{2}$.

	\nopagebreak
	We abbreviate $n/N = \beta$, and denote $n = (\frac{\lambda}{1+2\lambda}-\gamma)N$.
	Here, we allow $\gamma$ to be negative.
	Recall from \eqref{eq:recurrence solution integral}:
	\begin{align*}
		\pi_n
		=& -\frac{b}{\lambda^2}\binom{N-1}{n} (1+\lambda)^n
			\int_0^{\frac{\lambda}{1+2\lambda}}u^n(1-u)^{N-n-1}
				\Bigl(1 - \frac{1+2\lambda}{\lambda}u\Bigr)^{-\frac{2\lambda}{1+2\lambda}} \,du \\
		& -\frac{b}{\lambda}\binom{N-1}{n-1} (1+\lambda)^{n-1}
		\int_0^{\frac{\lambda}{1+2\lambda}}u^{n-1}(1-u)^{N-n-1}
			\Bigl(1 - \frac{1+2\lambda}{\lambda}u\Bigr)^{-\frac{2\lambda}{1+2\lambda}} \,du \\
		& -\frac{c - \frac{b}{\lambda}}{\lambda}\binom{N}{n} (1+\lambda)^n
			\int_0^{\frac{\lambda}{1+2\lambda}}u^n(1-u)^{N-n}
				\Bigl(1 - \frac{1+2\lambda}{\lambda}u\Bigr)^{-\frac{2\lambda}{1+2\lambda}} \,du \\
		={}& \frac{-b}{\lambda}
			(1+\lambda)^{n-1}
			\int_0^{\frac{\lambda}{1+2\lambda}} \binom{N}{n} t^n(1-t)^{N-n} S_n(t)
				\Bigl(
					1 - \frac{1+2\lambda}{\lambda}t
				\Bigr)^{-\frac{2\lambda}{1+2\lambda}} \,dt,
	\end{align*}
	where
	\begin{align*}
		S_n(t)
		&= \frac{1+\lambda}{\lambda}\frac{1-\beta}{1-t}
			+ \frac{\beta}{t}
			- \frac{1+\lambda}{\lambda}.
	\end{align*}
	Next, we shrink the integration interval to $(\frac{\lambda}{1+2\lambda}-\gamma_1,\frac{\lambda}{1+2\lambda})$, where $\gamma_1 = N^{-\frac{1}{2}+a}$.
	Note that this modification does not affect our argument;
		the lower bound remains unaffected,
		and for the upper bound, we observe, as in previous steps, that the integrand in the above expression is concentrated around $t = \frac{\lambda}{1+2\lambda}-\gamma$.
	On this interval,
		$S_n(t)$ is uniformly bounded and uniformly bounded away from $0$,
		thus allowing us to replace $S_n(t)$ with $1$.
	We will show that
	\begin{equation}\label{eq:vn asymptotic critical regime claim 1-1}
		e^{-2N\gamma_1}
		\ll \int_{\frac{\lambda}{1+2\lambda}-\gamma_1}^{\frac{\lambda}{1+2\lambda}}
			\binom{N}{n} t^n(1-t)^{N-n}
			\Bigl(
				1 - \frac{1+2\lambda}{\lambda}t
			\Bigr)^{-\frac{2\lambda}{1+2\lambda}} \,dt
	\end{equation}
	and
	\begin{equation}\label{eq:vn asymptotic critical regime claim 1-2}
		\int_{\frac{\lambda}{1+2\lambda}-\gamma_1}^{\frac{\lambda}{1+2\lambda}}
			\binom{N}{n} t^n(1-t)^{N-n}
			\Bigl(
				1 - \frac{1+2\lambda}{\lambda}t
			\Bigr)^{-\frac{2\lambda}{1+2\lambda}} \,dt
		\lesssim \max\{
				N^{-1+\frac{\lambda}{1+2\lambda}},
				N^{-\frac{1}{2}}|\gamma|^{\frac{1}{1+2\lambda}}
			\}.
	\end{equation}
	These claims together with \eqref{eq:b over pi0} conclude our proof since
	\[
		\frac{-b}{\lambda} = \pi_1 N (1 + o(1)).
	\]

	We first prove inequality \eqref{eq:vn asymptotic critical regime claim 1-1}.
	Denoting the integral on the right-hand side by $I_{N,n}$, we have 
	\begin{equation}\label{eq:critical regime lower bound step}
		I_{N,n}
		\ge \gamma_1 \binom{N}{n}
			\Bigl(
				\frac{\lambda}{1+2\lambda}-\gamma_1
			\Bigr)^n
			\Bigl(
				\frac{1}{1+2\lambda}
			\Bigr)^{N-n}.
	\end{equation}
	By Stirling's formula \cite[Equation (5.11.1)]{nist},
		we obtain
	\begin{align*}
		\log\binom{N}{n}
		&= - n\log \frac{n}{N}
			- (N-n)\log\frac{N-n}{N}
			+ O(\log N).
	\end{align*}
	Thus, by taking the logarithm on both sides of \eqref{eq:critical regime lower bound step} 
		and expanding in a Taylor series around $\frac{\lambda}{1+2\lambda}$, we arrive at 
	\[
		\log I_{N,n}
		\ge -\gamma_1 N + O(N\gamma_1^2),
	\]
	establishing our desired conclusion.

	To prove inequality \eqref{eq:vn asymptotic critical regime claim 1-2},
		we begin with the following quantitative local limit theorem for i.i.d.\ Bernoulli random variables:
	\[
		\biggl|
			\binom{N}{n}t^n(1-t)^{N-n} - \frac{1}{\sqrt{2\pi Nt(1-t)}}e^{-\frac{(n - Nt)^2}{2Nt(1-t)}}
		\biggr|
		\le \frac{0.516}{Nt(1-t)},
	\]
	as stated in Zolotukhin, Nagaev, and Chebotarev \cite[Lemma 5]{zolotukhin}.
	The error on the right-hand side can be neglected since
	\begin{align*}
		\int_{\frac{\lambda}{1+2\lambda}-\gamma_1}^\frac{\lambda}{1+2\lambda}
			\frac{0.516}{Nt(1-t)}
			\Bigl(
				1 - \frac{1+2\lambda}{\lambda}t
			\Bigr)^{-\frac{2\lambda}{1+2\lambda}} \,dt
		&= O(N^{-1})
			\int_{\frac{\lambda}{1+2\lambda}-\gamma_1}^\frac{\lambda}{1+2\lambda}
			\Bigl(
				1 - \frac{1+2\lambda}{\lambda}t
			\Bigr)^{-\frac{2\lambda}{1+2\lambda}} \,dt \\
		&= o(N^{-1+\frac{\lambda}{1+2\lambda}}).
	\end{align*}
	Therefore, it suffices to show 
	\[
		\int_{\frac{\lambda}{1+2\lambda}-\gamma_1}^\frac{\lambda}{1+2\lambda}
			\frac{1}{\sqrt{2\pi Nt(1-t)}}e^{-\frac{(n - Nt)^2}{2Nt(1-t)}}
			\Bigl(
				1 - \frac{1+2\lambda}{\lambda}t
			\Bigr)^{-\frac{2\lambda}{1+2\lambda}} \,dt
		\lesssim \max\{
			N^{-1+\frac{\lambda}{1+2\lambda}},
			N^{-\frac{1}{2}}|\gamma|^{\frac{1}{1+2\lambda}}
		\},
	\]
	or,
	\begin{equation*}
		\int_{\frac{\lambda}{1+2\lambda}-\gamma_1}^\frac{\lambda}{1+2\lambda}
			e^{- N(\beta - t)^2}
			\Bigl(
				\frac{\lambda}{1+2\lambda} - t
			\Bigr)^{-\frac{2\lambda}{1+2\lambda}} \,dt
		\lesssim \max\{
			N^{-\frac{1}{2}+\frac{\lambda}{1+2\lambda}},
			|\gamma|^{\frac{1}{1+2\lambda}}
		\}.
	\end{equation*}
	Note that the factor $\frac{1}{2t(1-t)}$ in the exponent has been replaced by a smaller constant $1$.
	By substituting $\frac{\lambda}{1+2\lambda} - t = u$, we obtain
	\begin{align*}
		\int_{\frac{\lambda}{1+2\lambda}-\gamma_1}^\frac{\lambda}{1+2\lambda}
			e^{- N(\beta - t)^2}
			\Bigl(
				\frac{\lambda}{1+2\lambda} - t
			\Bigr)^{-\frac{2\lambda}{1+2\lambda}} \,dt
		&= \int_0^{\gamma_1}
			e^{- N(\gamma - u)^2}
			u^{-\frac{2\lambda}{1+2\lambda}} \,du.
	\end{align*}
	We divide the integration interval into two parts:
	\begin{align*}
		\int_0^{\gamma_1}
			e^{- N(\gamma - u)^2}
			u^{-\frac{2\lambda}{1+2\lambda}} \,du
		&= \int_0^{\gamma_2}
				e^{- N(\gamma - u)^2}
				u^{-\frac{2\lambda}{1+2\lambda}} \,du
			+ \int_{\gamma_2}^{\gamma_1}
				e^{- N(\gamma - u)^2}
				u^{-\frac{2\lambda}{1+2\lambda}} \,du,
	\end{align*}
	where $\gamma_2 = \max\{N^{-\frac{1}{2}},|\gamma|\}$.
	For the first integral, observe that
	\begin{align*}
		\int_0^{\gamma_2}
				e^{- N(\gamma - u)^2}
				u^{-\frac{2\lambda}{1+2\lambda}} \,du
		&\le \int_0^{\gamma_2} u^{-\frac{2\lambda}{1+2\lambda}} \,du
		\lesssim \max\{
				N^{-\frac{1}{2}+\frac{\lambda}{1+2\lambda}},
				|\gamma|^{\frac{1}{1+2\lambda}}
			\}.
	\end{align*}
	For the second integral, we assert that
	\begin{equation*}
		\int_{\gamma_2}^{\gamma_1}
				e^{- N(\gamma - u)^2}
				u^{-\frac{2\lambda}{1+2\lambda}} \,du
		\le N^{\frac{\lambda}{1+2\lambda}}
			\int_{\gamma_2}^\infty
				e^{- N(\gamma - u)^2} \,du 
		\le N^{\frac{\lambda}{1+2\lambda}}
			\int_0^\infty
				e^{- N u^2} \,du 
		\lesssim N^{-\frac{1}{2}+\frac{\lambda}{1+2\lambda}},
	\end{equation*}
	and the proof is complete.
\end{proof}

\subsection{Basic Properties of the Energy Landscape}\label{subsection:landscape property}
In this subsection,
	we establish some basic properties of
	the quasi-stationary measure of the contact process on stars.
These properties are crucial for subsequent computations within the potential theoretic framework.

\begin{lemma}
	\label{lem:mass at metastable}
	Let $v_n = \mu(1,n)$ denote the scaled quasi-stationary measure at states with an infected hub.
	Define $m = \lfloor\frac{\lambda}{1+\lambda}N\rfloor$,
		which represents the number of infected leaves at the metastable state.
	Then, as $N \to \infty$, the mass at the metastable state satisfies
	\begin{equation*}
		v_m
		\simeq \frac{1+\lambda}{\sqrt{2\pi\lambda N}}
			\Bigl(
				\frac{1+\lambda}{\lambda}
			\Bigr)^{N + \frac{2}{1+2\lambda}}.
	\end{equation*}
\end{lemma}
\begin{proof}
	By the asymptotic formulas for $v_n$ in Theorem \ref{thm:stationary asymptotic high},
		it suffices to show that
	\begin{equation}\label{eq:mass at metastable claim}
		v_m^{\mathrm{high}}
		= \binom{N}{m}\lambda^{m-N} f_{N,\lambda}(m)
		= \frac{1+\lambda}{\sqrt{2\pi\lambda N}}
			\Bigl(
				\frac{1+\lambda}{\lambda}
			\Bigr)^{N + \frac{2}{1+2\lambda}}
			(1+o(1)).
	\end{equation}

	By Stirling's formula \cite[Equation (5.11.1)]{nist},
		we obtain
	\begin{equation*}
		\binom{N}{m}
		= \frac{1}{\sqrt{2\pi\lambda N}}
			\frac{(1+\lambda)^{N+1}}{\lambda^{m}}
			(1+o(1)).
	\end{equation*}
	Also, a computation analogous to \eqref{eq:alternative gn proof} reveals that
		for $n \in [\frac{\lambda}{1+2\lambda}N+\varepsilon N,N]$, where $\varepsilon > 0$,
		we have
	\begin{equation*}
		f_{N,\lambda}(n)
		= \Bigl(
			\frac{1+2\lambda}{1+\lambda}\frac{n}{N} - \frac{\lambda}{1+\lambda}
		\Bigr)^{-\frac{1}{1+2\lambda}}
		(1+O(N^{-1})).
	\end{equation*}
	Summing up, we conclude that \eqref{eq:mass at metastable claim} holds.
\end{proof}

\begin{lemma}
	\label{lem:total mass}
	Let $u_n = \mu(0,n)$ and $v_n = \mu(1,n)$ denote the scaled quasi-stationary measure.
	Let $\varepsilon > 0$ be given sufficiently small.
	Define $m = \lfloor\frac{\lambda}{1+\lambda}N\rfloor$,
		which represents the number of infected leaves at the metastable state,
		and set $R = N^{\frac{1}{2}+\varepsilon}$.
	Then, as $N \to \infty$, the total mass of the process satisfies
	\begin{equation}\label{eq:total mass}
 		Z_{N,\lambda}
		= \sum_{n = 0}^N (u_n + v_n)
		\simeq \sum_{|n - m| < R} v_n 
		\simeq \Bigl(
				\frac{1+\lambda}{\lambda}
			\Bigr)^{N + \frac{2}{1+2\lambda}},
 	\end{equation}
	where the second summation is taken over all integers $n$ satisfying $|n - m| < R$.
\end{lemma}
\begin{proof}
	We begin by proving the relation 
	\[
		\sum_{n=0}^N v_n
		\simeq \sum_{|n - m| < R} v_n 
		\simeq \frac{\sqrt{2\pi\lambda N}}{1+\lambda} v_m .
	\]
	The right-hand side of the above relation 
		agrees with the right-hand side of \eqref{eq:total mass}
		by Lemma \ref{lem:mass at metastable}.
	Write $n = m + k$.
	By the asymptotic formulas for $v_n$
		in Theorems \ref{thm:stationary asymptotic high}
		and \ref{thm:stationary asymptotic low},
		$v_n/v_m$ decays rapidly as $N\to \infty$ if $k \ge R$.
	Hence, we may neglect integers $n$ satisfying $|n - m| \ge R$,
		and suppose that $k < R$.
	Similarly to the proof of Lemma \ref{lem:mass at metastable},
		we obtain
	\begin{align*}
		v_n
		&= v_n^{\mathrm{high}}(1+o(1)) \\
		&= \frac{1}{\sqrt{2\pi N}}
			\frac{1}{(\frac{n}{N})^{n+\frac{1}{2}}(\frac{N-n}{N})^{N-n+\frac{1}{2}}}
			\lambda^{n-N}
			\Bigl(
				\frac{1+2\lambda}{1+\lambda}\frac{n}{N} - \frac{\lambda}{1+\lambda}
			\Bigr)^{-\frac{1}{1+2\lambda}}
			(1+o(1)) \\
		&= v_m
			\frac{1}{
				(1 + \frac{1+\lambda}{\lambda}\frac{k}{N})^{n}
				(1-(1+\lambda)\frac{k}{N})^{N-n}
			}
			(1+o(1)).
	\end{align*}
	It can also be verified that
	\[
		\frac{1}{
			(1 + \frac{1+\lambda}{\lambda}\frac{k}{N})^{n}
			(1-(1+\lambda)\frac{k}{N})^{N-n}
		}
		= \exp\Bigl[
			-\frac{(1+\lambda)^2}{2\lambda}\frac{k^2}{N}
		\Bigr]
		\Bigl[
			1 + O\Bigl(\frac{k}{N}\Bigr)
		\Bigr]
	\]
	by taking the logarithm on both sides and expanding in a Taylor series.
	Thus, it follows that
	\begin{align*}
		\sum_{|n - m| < R} v_n
		&\simeq v_m \sum_{|k| < R} \exp\Bigl[
			-\frac{(1+\lambda)^2}{2\lambda}\frac{k^2}{N}
		\Bigr] 
		\simeq \frac{2\pi\lambda N}{1+\lambda} v_m
			\erf\Bigl(
				\frac{1+\lambda}{\sqrt{2\lambda}}\frac{R}{\sqrt{N}}
			\Bigr)
		\simeq \frac{2\pi\lambda N}{1+\lambda} v_m ,
	\end{align*}
	where $\erf(x)$ denotes the error function.

	It remains to prove that
		the sum of the terms $u_n$ is negligible with respect to the sum of the terms $v_n$.
	By the relation \eqref{eq:recurrence vn-2} between the terms $u_n$ and $v_n$, we have 
	\begin{align*}
		\sum_{n=0}^N u_n
		&= u_0 + \lambda N v_0
			+ \sum_{n=2}^N \Bigl(
				\lambda\frac{N-n}{n+1} - 1
			\Bigr) v_n \\
		&= \sum_{|n-m|<R}	\Bigl(
			\lambda\frac{N-n}{n+1} - 1
			\Bigr) v_n (1+o(1))
		= O(N^{-\frac{1}{2}+\varepsilon}) \sum_{|n-m|<R} v_n,
	\end{align*}
	establishing our claim.
\end{proof}

\begin{lemma}
	\label{lem:sum of scaled vn}
	Let $v_n = \mu(1,n)$ denote the scaled quasi-stationary measure at states with an infected hub.
	Then, as $N \to \infty$, we have 
	\begin{equation}\label{eq:sum of scaled vn}
		\sum_{n = 0}^N \frac{v_n}{(1+\lambda)^{n}}
		= \lambda N C_{N,\lambda} (1 + O(N^{-1})).
	\end{equation}
\end{lemma}
The proof of Lemma \ref{lem:sum of scaled vn} is given in Section \ref{section:proof}.

\section{Potential Theory for Non-Reversible Markov Chains}\label{section:theoretical background}
In this section,
	we introduce basic terminologies and
	review a potential theoretic approach for estimating the mean extinction time.
These concepts and methods are employed in the next section.
We refer the reader to \cite{seo-arxiv} for more details.

Throughout this section,
let $(X(t))_{t\ge 0}$ be an irreducible continuous-time Markov process taking values in a finite set $\mathcal{H}$,
with jump rates $r: \mathcal{H}\times\mathcal{H} \to [0,\infty)$
and stationary distribution $\mu$.

\subsection{Basic Terminology}
Since $\mathcal{H}$ is a finite set, the space $L^2(\mu)$ consists of all real functions on $\mathcal{H}$.
The generator $\mathscr{L}$ is an operator acting on $f \in L^2(\mu)$ by
\begin{equation*}
	(\mathscr{L}f)(x)
	= \sum_{y\in \mathcal{H}} r(x,y)(f(y)-f(x)),\qquad
	x\in\mathcal{H}.
\end{equation*}
Then $\mathscr{L}$ defines a positive semi-definite quadratic form on $L^2(\mu)$ given by
\begin{equation*}
\begin{split}
	\mathscr{D}(f)
	&= \langle f, -\mathscr{L}f\rangle_\mu
	= \frac{1}{2} \sum_{x\in \mathcal{H}} \sum_{y\in\mathcal{H}} \mu(x)r(x,y) [f(y)-f(x)]^2
\end{split}
\end{equation*}
where $f\in L^2(\mu)$, called the \emph{Dirichlet form}.

The time-reversed process of $(X(t))_{t\ge 0}$, called the \emph{adjoint process}, is
the continuous-time Markov process $(X^\dagger(t))_{t\ge 0}$ on $\mathcal{H}$
with rates
\begin{equation*}
	r^\dagger(x,y)
	= \frac{\mu(y)r(y,x)}{\mu(x)}, \qquad
	x,y\in\mathcal{H}.
\end{equation*}
The generator $\mathscr{L}^\dagger$ of the adjoint process, given by
\begin{equation*}
	(\mathscr{L}^\dagger f)(x)
	= \sum_{y\in\mathcal{H}} r^\dagger(x,y) (f(y)-f(x)), \qquad
	x\in\mathcal{H}
\end{equation*}
for $f\in L^2(\mu)$, is indeed the adjoint operator to $\mathscr{L}$, that is,
\[
	\langle f,\mathscr{L}g\rangle_\mu = \langle \mathscr{L}^\dagger f, g\rangle_\mu
\]
for $f,g\in L^2(\mu)$.
The process is said to be \emph{reversible} if $\mathscr{L}^\dagger = \mathscr{L}$.

Define the \emph{symmetrized process} of $(X(t))_{t\ge 0}$ as
the continuous-time Markov process $(X^s(t))_{t\ge 0}$ on $\mathcal{H}$ with rates
\begin{equation*}
	r^s(x,y)
	= \frac{1}{2\mu(x)}[\mu(x)r(x,y) + \mu(y)r(y,x)],\qquad
	x,y\in\mathcal{H}.
\end{equation*}
Note that $\mu$ is the stationary distribution for the reversible process $(X^s(t))_{t\ge 0}$.

For nonempty disjoint subsets $A$ and $B$ of $\mathcal{H}$,
	define the \emph{equilibrium potential} between $A$ and $B$ 
	with respect to the process $(X(t))_{t\ge 0}$ 
	as the function $h_{A,B}: \mathcal{H} \to [0,1]$ given by
\begin{equation*}
	h_{A,B}(x) = \PP_x[\tau_A < \tau_B], \qquad
	x \in \mathcal{H},
\end{equation*}
where $\PP_x$ denotes the law of the process $(X(t))_{t\ge 0}$ starting from $x$.
It can be immediately checked that
\[
	h_{B,A} = 1 - h_{A,B},
\]
and
\[
	\begin{cases}  \smallskip
		h_{A,B} \equiv 1	&	\text{on } A, \\  \smallskip
		h_{A,B} \equiv 0	&	\text{on } B,\text{ and } \\  \smallskip
		\mathscr{L}h_{A,B} \equiv 0	&	\text{on } (A\cup B)^c.
	\end{cases}
\]
Denote the equilibrium potential with respect to the adjoint process $(X^\dagger(t))_{t\ge 0}$
	by $h_{A,B}^\dagger$.
The \emph{capacity} between $A$ and $B$ with respect to the process $(X(t))_{t\ge 0}$ is defined as
\begin{equation*}
	\CAP(A,B) = \mathscr{D}(h_{A,B}).
\end{equation*}
Note that we have 
\[
	\CAP(A,B) = \mathscr{D}(h_{A,B}) = \mathscr{D}(h_{B,A}) = \CAP(B,A).
\]

The capacity satisfies two important basic properties.
See \cite[Propositions 1.9, 1.10]{seo-arxiv} for more details.
\begin{proposition}\label{prop:time symmetry of capacity}
	Let $A$ and $B$ be two nonempty disjoint subsets of $\mathcal{H}$.
	Then, we have 
	\begin{equation*}
		\CAP(A,B) = \CAP^\dagger(A,B),
	\end{equation*}
	where $\CAP^\dagger$ denotes the capacity with respect to the adjoint process.
\end{proposition}

\begin{proposition}\label{prop:monotonicity of capacity}
	Let $A'$ and $B'$ be two nonempty disjoint subsets of $\mathcal{H}$,
		and $A$ and $B$ be nonempty subsets of $A'$ and $B'$, respectively.
	Then, we have 
	\begin{equation*}
		\CAP(A,B) \le \CAP(A',B').
	\end{equation*}
\end{proposition}

Given a process, we can represent its mean hitting times
	in terms of capacities, equilibrium potentials, and the stationary distribution.

\begin{proposition}[Mean hitting time formula, {\cite[Equation (1.32)]{seo-arxiv}}]\label{prop:mean hit time formula}
	Let $x,y \in \mathcal{H}$ be two distinct states.
	Then, we have 
	\begin{equation*}
		\EE_x [\tau_y]
		= \frac{1}{\CAP(x,y)} \sum_{z\in\mathcal{H}} h_{x,y}^\dagger(z)\mu(z).
	\end{equation*}
\end{proposition}

In general, it is difficult to compute the equilibrium potential $h_{A,B}$ accurately.
Hence, the following rough estimate for $h_{A,B}$ provides a useful bound.
See \cite[Proposition 1.16]{seo-arxiv} for the proof.
\begin{proposition}\label{prop:equilibrium potential estimate}
	Let $A$ and $B$ be two nonempty disjoint subsets of $\mathcal{H}$.
	Then, we have 
	\begin{equation*}
		1 - \frac{\CAP(x,B)}{\CAP(x,A\cup B)}
		\le h_{A,B}(x)
		\le \frac{\CAP(x,A)}{\CAP(x,A\cup B)}
		\text{ for all } x\in (A\cup B)^c.
	\end{equation*}
\end{proposition}

Next, we introduce the flow structure associated with the Markov process.
For two sites $x$ and $y$ in $\mathcal{H}$,
	we write $x \sim y$ if $r(x,y) + r(y,x) > 0$.
Note that $x \sim y$ if and only if $y \sim x$.
Define the set of directed edges by
\begin{equation*}
	\mathfrak{E}
	= \{(x,y) \in \mathcal{H}\times\mathcal{H} : x \sim y\}.
\end{equation*}
A \emph{flow} on $\mathcal{H}$ is a function $\phi: \mathfrak{E} \to \RR$ that is anti-symmetric,
in the sense that
\[
	\phi(x,y) = - \phi(y,x)
	\text{ for all }(x,y) \in \mathfrak{E}.
\]
We denote the space of flows by $\mathfrak{F}$.
Define the \emph{conductance} between two sites $x$ and $y$ as
\begin{equation*}
	c(x,y) = \mu(x)r(x,y),\qquad
	x,y\in\mathcal{H},
\end{equation*}
and consider the symmetrized conductance
\begin{equation*}
	c^s(x,y)
	= \frac{1}{2}[c(x,y) + c(y,x)],\qquad
	x,y\in \mathcal{H},
\end{equation*}
satisfying $c^s(x,y) = c^s(y,x)$.
Then, we define an $L^2$-structure on the flow space by
\begin{equation*}
	\langle \phi,\psi \rangle_\mathfrak{F}
	= \frac{1}{2} \sum_{(x,y)\in\mathfrak{E}} \frac{\phi(x,y)\psi(x,y)}{c^s(x,y)},\qquad
	\phi,\psi\in\mathfrak{F}.
\end{equation*}
The flow norm is defined as $\|\phi\|_\mathfrak{F} = \langle\phi,\phi\rangle_\mathfrak{F}^{1/2}$.

For a flow $\phi$, the \emph{divergence} of $\phi$ at a site $x\in\mathcal{H}$ is defined by
\begin{equation*}
	(\DIV \phi)(x)
	= \sum_{y: x\sim y} \phi(x,y).
\end{equation*}
For $A \subseteq \mathcal{H}$, define
\begin{equation*}
	(\DIV \phi)(A)
	= \sum_{x \in A} (\DIV \phi)(x).
\end{equation*}
The flow $\phi$ is said to be \emph{divergence-free} at $x \in \mathcal{H}$ if $(\DIV\phi)(x) = 0$,
and divergence-free on $A \subseteq \mathcal{H}$ if it is divergence-free at all $x\in A$.

Given a function $f: \mathcal{H} \to \RR$, we define three associated flows as follows:
\begin{equation*}
\begin{split}
	\Phi_f(x,y) &= f(y)c(y,x) - f(x)c(x,y), \\
	\Phi_f^*(x,y) &= f(y)c(x,y) - f(x)c(y,x), \\
	\Psi_f(x,y) &= c^s(x,y)[f(y)-f(x)] = (1/2)(\Phi_f + \Phi_f^*)(x,y).
\end{split}
\end{equation*}
Then, it holds that 
\begin{equation}\label{eq:div of associated flows}
	(\DIV \Phi_f)(x) = \mu(x) (\mathscr{L}^\dagger f)(x) \quad\text{and}\quad 
	(\DIV \Phi_f^*)(x) = \mu(x) (\mathscr{L} f)(x)
\end{equation}
for all $x\in\mathcal{H}$.

\subsection{Dirichlet and Thomson Principles}
Variational principles are useful tools for estimating the capacity of a process.
In this subsection, we introduce two variational principles:
	the Dirichlet principle and the Thomson principle,
	which provide upper and lower bounds for the capacity, respectively.
For a deeper discussion of the principles, 
	we refer the reader to \cite[Theorem 3.2]{seo-arxiv}.

For nonempty and disjoint subsets $A$ and $B$ of $\mathcal{H}$, and real numbers $a$ and $b$,
let $\mathfrak{C}_{a,b}(A,B)$ be the set of all real-valued functions $f$ on $\mathcal{H}$
such that $f\vert_A \equiv a$ and $f\vert_B \equiv b$.

\begin{theorem}[Dirichlet principle]\label{thm:dirichlet}
	Let $(X(t))_{t\ge 0}$ be a continuous-time Markov process on a finite set $\mathcal{H}$,
	and $A, B \subseteq \mathcal{H}$ be nonempty and disjoint.
	Then, we have 
	\begin{equation*}
		\CAP(A,B)
		= \inf_{f\in\mathfrak{C}_{1,0}(A,B), \phi\in\mathfrak{F}}
		\biggl\{ \|\Phi_f - \phi \|^2 - 2\sum_{x\in\mathcal{H}} h_{A,B}(x) (\DIV\phi)(x) \biggr\},
	\end{equation*}
	and
	\begin{equation*}
		(f,\phi)
		= \Bigl(\frac{1}{2}(h_{A,B} + h_{A,B}^\dagger),
			\frac{1}{2}(\Phi_{h_{A,B}^\dagger} - \Phi_{h_{A,B}}^*)\Bigr)
	\end{equation*}
	is the unique minimizer.
\end{theorem}

\begin{theorem}[Thomson principle]\label{thm:thomson}
	Let $(X(t))_{t\ge 0}$ be a continuous-time Markov process on a finite set $\mathcal{H}$,
	and $A, B \subseteq \mathcal{H}$ be nonempty and disjoint.
	Then, we have 
	\begin{equation*}
		\CAP(A,B)
		= \sup_{g\in\mathfrak{C}_{0,0}(A,B), \psi\in\mathfrak{F}\setminus\{0\}}
			\frac{1}{\|\Phi_g - \psi\|^2}
			\biggl[ \sum_{x\in\mathcal{H}} h_{A,B}(x)(\DIV\psi)(x) \biggr]^2,
	\end{equation*}
	and constant multiples of
	\begin{equation*}
		(g,\psi)
		= \Bigl(\frac{1}{2\CAP(A,B)}(h_{A,B}^\dagger - h_{A,B}),
			\frac{1}{2\CAP(A,B)}(\Phi_{h_{A,B}^\dagger} + \Phi_{h_{A,B}}^*)\Bigr)
	\end{equation*}
	are maximizers.
\end{theorem}

\begin{remark}
	Both principles involve the expression
	\[
		\sum_{x\in\mathcal{H}} h_{A,B}(x)(\DIV \phi)(x),
	\]
	which admits the decomposition 
	\[
		(\DIV\phi)(A) + \sum_{x\in (A\cup B)^c} h_{A,B}(x)(\DIV \phi)(x).
	\]
	Note that if we select the test function and flow as the corresponding minimizer or maximizer,
		then the second term vanishes.
	Moreover, the first term equals $0$ for the Dirichlet principle
		and equals $1$ for the Thomson principle.
	In practice, to construct an effective test function and flow,
		one typically partitions $(A \cup B)^c$ into two subsets, 
		$\mathcal{C}_1$ and $\mathcal{C}_2$, such that
		the test flow is approximately divergence-free on $\mathcal{C}_1$,
		while the function $h_{A,B}$ is small on $\mathcal{C}_2$.
\end{remark}

\subsection{Trace Processes}\label{subsection:trace process}
In this subsection, we briefly introduce the notion of the trace process.
For a deeper treatment of this theory, we refer the reader to \cite{beltran-12}.

Let $F$ be a proper subset of $\mathcal{H}$.
The \emph{trace process} of $(X(t))_{t\ge 0}$ on $F$ is
	defined as the process obtained by ignoring the time spent by $(X(t))_{t\ge 0}$ outside the set $F$.
More precisely,
	define $(\mathcal{T}_t)_{t\ge 0}$ to be the time that 
	$(X(t))_{t\ge 0}$ spends in the set $F$ during the time interval $[0,t]$,
	and let $(\mathcal{S}_t)_{t\ge 0}$ be the generalized inverse of $(\mathcal{T}_t)_{t\ge 0}$.
Then, the trace process $(X^F(t))_{t\ge 0}$ is given by $X^F(t) = X(\mathcal{S}_t)$,
which is well-defined and almost surely takes values in $F$.

The stationary distribution of the trace process is the restriction of $\mu$ to the set $F$, that is,
\begin{equation*}
	\mu_F = \frac{1}{\mu(F)}\mu\vert_F.
\end{equation*}
Denoting the capacity with respect to the trace process by $\CAP_F$, we have the relation 
\begin{equation}\label{eq:capacity for trace process}
	\mu(F) \CAP_F(A,B) = \CAP(A,B)
\end{equation}
for all nonempty disjoint subsets $A$ and $B$ of $F$.

\section{Proof of the Eyring--Kramers Law}\label{section:proof}
In this section, we prove the Eyring--Kramers law for
	the mean extinction time of the contact process on star graphs
	by estimating the capacity associated with the process.

Following the approach similar to those found in \cite{cator} and \cite{chatterjee},
	we consider the trace process of the regenerative process $(o_t,n_t)_{t\ge 0}$, 
	restricted to the set
\begin{equation*}
	F = \{(0,0)\} \cup \{(1,n): 0\le n\le N\}.
\end{equation*}
In other words, we disregard the time spent by the process when the hub is healthy,
	except when the process is at the stable state $(0,0)$.
It turns out that the jump rates of the trace process can be explicitly determined.
When the hub is healthy and there are $n > 0$ infected leaves,
	the subsequent transition of the contact process is
	either to the state $(1,n)$ with probability $\frac{\lambda}{1+\lambda}$
	or to the state $(0,n-1)$ with probability $\frac{1}{1+\lambda}$.
Hence, once the hub becomes healthy,
	the number $H$ of leaves that recover before the hub is reinfected follows 
	a truncated geometric distribution:
\[
	\PP[H = j]
	= \begin{cases} \smallskip
		\frac{\lambda}{(1+\lambda)^{j+1}}	&	\text{for } 0\le j\le n-1, \\ \smallskip
		\frac{1}{(1+\lambda)^{n}}			&	\text{for } j = n.
	\end{cases}
\]
Consequently, the jump rate $r_F(x,y)$ of the trace process on $F$ is given by
\begin{equation*}
	\begin{cases} \smallskip
		r_F((1,n),(1,n+1)) = \lambda(N-n)
			& \text{for } 0\le n\le N, \\ \smallskip
		r_F((1,n),(1,n-1)) = n + \frac{\lambda}{(1+\lambda)^2}
			& \text{for } 1\le n\le N, \\ \smallskip
		r_F((1,n),(1,n-j)) = \frac{\lambda}{(1+\lambda)^{j+1}}
			& \text{for } 1\le n\le N \text{ and } 2 \le j \le n-1, \\ \smallskip
		r_F((1,n),(0,0)) = \frac{1}{(1+\lambda)^n}
			& \text{for } 0\le n\le N, \\ \smallskip
		r_F((0,0),(1,0)) = \alpha.
	\end{cases}
\end{equation*}

Lemma \ref{lem:sum of scaled vn} follows immediately from the definition of the trace process.
\begin{proof}[Proof of Lemma \ref{lem:sum of scaled vn}]
	By the stationarity condition at the state $(0,0)$ for the trace process on the set $F$, we have 
	\[
		\sum_{n=0}^{N} \frac{\mu(1,n)}{(1+\lambda)^n} = \alpha \mu(0,0).
	\]
	Thus, the lemma follows directly from \eqref{eq:u0 asymptotic}.
\end{proof}

Recall from Lemma \ref{lem:total mass} that 
	the mass of the set $F$ converges to $1$ as $N \to \infty$.
For simplicity, throughout this section, 
	we treat the quasi-stationary distribution $\nu$ 
	as if it were the stationary distribution of the trace process.
It is also important to note that 
	a sharp estimate for the capacity of the trace process on $F$ 
	provides a corresponding estimate for the original process.

\subsection{Capacity Estimate}
In this subsection, we apply variational principles
	to derive a sharp estimate for the capacity of the process.

We consider the capacity between the all-healthy state $(0,0)$ 
	and a state $x=(1,n)$ with an infected hub.
From a potential theoretic viewpoint, 
	we can crudely estimate the capacity between two states 
	(on an exponential scale)
	as the ratio of the minimum to the maximum stationary measure 
	evaluated along the most probable path connecting these two states.
In our context, 
	this estimate corresponds to 
	the quasi-stationary measure at the metastable state 
	divided by the maximum of the quasi-stationary measures 
	at the states $(0,0)$ and $(1,n)$.
This is because the most probable path from $(0,0)$ to $(1,n)$ 
	necessarily passes through the metastable state.
Consequently, the capacity between the states $(0,0)$ and $(1,n)$ is approximated, 
	on an exponential scale, 
	by $\mu(1,\lfloor\frac{\lambda}{1+\lambda}N\rfloor)/\mu(0,0)$, 
	provided that $\mu(1,n)$ remains negligible compared to $\mu(0,0)$.

By Theorems \ref{thm:stationary asymptotic high} and \ref{thm:stationary asymptotic low}, 
	together with Stirling's formula \cite[Equation (5.11.1)]{nist}, 
	we obtain, 
	for $\frac{\lambda}{1+2\lambda} < t \le 1$, 
	the asymptotic relation 
\[
	\frac{1}{N}\log \frac{\mu(0,0)}{\mu(1,\lfloor tN\rfloor)}
	= s(\lambda,t)+o(1),
\]
where 
\[
	s(\lambda,t)
	= \log \frac{1+2\lambda}{\lambda(1+\lambda)}
		+ t\log t + (1-t)\log (1-t) + (1-t)\log\lambda.
\]
We adopt the convention $0\log 0 = 0$.
Note that $s(\varphi,1)=0$, where $\varphi = \frac{1+\sqrt{5}}{2}$. 
It can be readily verified that if $\lambda \le \varphi$, 
	then there exists a unique $\widetilde{w}$ satisfying 
	$\frac{\lambda}{1+2\lambda}<\widetilde{w}\le 1$ and $s(\lambda,\widetilde{w})=0$.
Moreover, this $\widetilde{w}$ is greater than $\frac{\lambda}{1+\lambda}$.
We define a value $w$ with 
	$\frac{\lambda}{1+\lambda}<w<\widetilde{w}$, 
	chosen sufficiently close to $\widetilde{w}$ so that 
	$s(\lambda,w) > - 2 w \log(1+\lambda)$
	when $\lambda \le \varphi$. 
For $\lambda > \varphi$, we set $w = 1$.

\begin{theorem}[Capacity estimate]\label{thm:capacity estimate}
	Let $\varepsilon > 0$ be given.
	Let $w$ be as defined above, and let $W = \lfloor wN \rfloor$.
	Then, for $n \in [\varepsilon N, W]$,
		we have
	\begin{equation}\label{eq:capacity estimate}
	\begin{split}
		\CAP((0,0),(1,n))
		&= (1+2\lambda)
			\Bigl(\frac{\lambda}{1+\lambda}\Bigr)^{\frac{2}{1+2\lambda}}
			B\Bigl(\tfrac{1}{1+2\lambda},N\Bigr)^{-1}
			\Bigl(\frac{1+2\lambda}{(1+\lambda)^2}\Bigr)^{N}
			(1+o(1))
	\end{split}
	\end{equation}
	as $N\to\infty$, where the error term $o(1)$ is uniform in $n$.
\end{theorem}

To prove the theorem stated above,
	we construct effective test functions and flows
	for use in variational principles associated with the trace process.
Initially, we define two functions $h$ and $h^\dagger$,
	which are designed to approximate the equilibrium potentials of the trace process and its adjoint 
	between the stable state $(0,0)$ and the state $(1,n)$.
Next, we introduce tentative test functions and flows:
\begin{equation}\label{eq:test function flow}
	(f,\phi) = \Bigl(\frac{1}{2}(h + h^\dagger),
		\frac{1}{2}(\Phi_{h^\dagger} - \Phi_{h}^*)\Bigr) \quad\text{and}\quad
	(g,\psi) = \Bigl(\frac{1}{2}(h^\dagger - h),
		\frac{1}{2}(\Phi_{h^\dagger} + \Phi_{h}^*)\Bigr),
\end{equation}
which are analogous to the extremizers
	appearing in Theorems \ref{thm:dirichlet} and \ref{thm:thomson}.
Finally, we slightly modify these test flows $\phi$ and $\psi$ to ensure that
	the resulting flows are divergence-free, except at the states $(0,0)$ and $(1,n)$.
We now provide heuristic arguments 
	that motivate our choice of test functions and flows.

In the case of the function $h$, 
	since an infected hub transmits infection simultaneously to all $N$ leaves,
	the process is highly unlikely to reach the stable state before hitting the state $(1,n)$.
Thus, it is natural to define 
\begin{equation*}
	h(x) = \begin{cases} \smallskip
		1	&	\text{if } x = (0,0), \\  \smallskip
		0	&	\text{otherwise}
	\end{cases}
\end{equation*}
for $x \in F$.

Next, consider the adjoint $(X_F^\dagger(t))_{t\ge 0}$ of the trace process,
which has jump rates
\[
	r_F^\dagger(x,y)
	= \frac{\mu(y)}{\mu(x)} r_F(y,x),\qquad
	x,y\in F.
\]
Suppose that $x = (1,k)$, and $1 \ll k \ll N$.
By Theorem \ref{thm:stationary asymptotic low}, as $N \to \infty$,
\begin{align*}
	r_F^\dagger((1,k),(1,k-1))
		&= \frac{\lambda}{1+\lambda}N (1 + o(1)), \\
	r_F^\dagger((1,k),(1,k+1))
		&= (1+\lambda) k (1 + o(1)).
\end{align*}
Moreover, for long jumps,
	if $k + 1 < k + j \le N$, we have 
\begin{equation*}
	r_F^\dagger((1,k),(1,k+j))
	= \frac{\mu(1,k+j)}{\mu(1,k)}
		\frac{\lambda}{(1+\lambda)^{j+1}}.
\end{equation*}
Recall from Lemma \ref{lem:sum of scaled vn} that we have 
\[
	\sum_{j=2}^{N-k} \mu(1,k+j) (1+\lambda)^{-(k+j)}
	= \lambda N C_{N,\lambda} (1+o(1)),
\]
since
\[
	\mu(1,k) (1+\lambda)^{-k} = C_{N,\lambda}(1 + o(1))
\]
provided $k \ll N$.
Hence, the process performs a long jump at rate 
\begin{align*}
	\sum_{j=2}^{N-k} r_F^\dagger((1,k),(1,k+j))
	&= \sum_{j=2}^{N-k}
		\frac{\lambda}{1+\lambda}
		\frac{\mu(1,k+j)(1+\lambda)^{-(k+j)}}{\mu(1,k) (1+\lambda)^{-k}}
	= \frac{\lambda^2}{1+\lambda} N (1+o(1)).
\end{align*}
Therefore, roughly speaking, if the adjoint process is positioned at $(1,k)$,
	the value of $k$ decreases by $1$ at an approximate rate $\frac{\lambda}{1+\lambda}N$,
	whereas it increases by a large amount at an approximate rate $\frac{\lambda^2}{1+\lambda} N$.
Interpreting the long jumps as transitions directly to the state $(1,n)$, 
	each jump of the adjoint process either brings the system one step closer to the state $(0,0)$ with probability $\frac{1}{1+\lambda}$, 
	or results in a transition to the state $(1,n)$ with probability $\frac{\lambda}{1+\lambda}$.
Consequently, it is plausible to select an approximately geometric form for the function $h^\dagger$: 
\begin{equation*}
	h^\dagger(x) = \begin{cases}  \smallskip
		1	&	\text{if } x = (0,0), \\  \smallskip
		(1+\lambda)^{-k}	&	\text{if } x = (1,k),\ 0\le k\le R_1, \\  \smallskip
		0	&	\text{otherwise}
	\end{cases}
\end{equation*}
for some $1 \ll R_1 \ll N$, say $R_1 = \lfloor N^q\rfloor$ for some small $q > 0$.

The divergence of the associated flows of $h$ and $h^\dagger$ can be directly computed 
	from \eqref{eq:div of associated flows} 
	by applying \eqref{eq:u0 asymptotic} and Lemma \ref{lem:sum of scaled vn}.
\begin{lemma}\label{lem:div of associated flows}
	Let $0< q < 1$, $R_1 = \lfloor N^q \rfloor$,
		and let $h$ and $h^\dagger$ be as above. 
	Then for $x \in F$, as $N \to \infty$, we have 
	\begin{equation}\label{eq:div of phi h star}
		(\DIV \Phi_{h}^*)(x)
		= \begin{cases} \smallskip
			-\frac{1}{Z_{N,\lambda}} \lambda N C_{N, \lambda} (1 + o(1))
				& \text{if } x = (0,0), \\ \smallskip
			\frac{1}{Z_{N,\lambda}} \mu(1,k) (1+\lambda)^{-k}
				& \text{if } x = (1,k),\ 0\le k\le N,
		\end{cases}
	\end{equation}
	and 
	\begin{equation}\label{eq:div of phi h dagger}
		(\DIV \Phi_{h^\dagger})(x)
		= \begin{cases} \smallskip
			-\frac{1}{Z_{N,\lambda}} \lambda N C_{N, \lambda} (1 + o(1))
				& \text{if } x = (0,0), \\ \smallskip
			\frac{1}{Z_{N,\lambda}} C_{N, \lambda} O(N^q)
				& \text{if } x = (1,k),\ 0\le k\le R_1, \\ \smallskip
			\frac{1}{Z_{N,\lambda}} \lambda N C_{N, \lambda} (1 + o(1))
				& \text{if } x = (1,R_1 + 1), \\ \smallskip
			0	& \text{otherwise}.
		\end{cases}
	\end{equation}
\end{lemma}

We now modify the test flows $\phi$ and $\psi$ defined in \eqref{eq:test function flow}
	to ensure that these flows become divergence-free, except at the sites $(0,0)$ and $(1,n)$.
Heuristically, we anticipate that the equilibrium potential of the trace process
	decays rapidly for states distant from the stable state $(0,0)$.
Thus, to obtain accurate estimates from the variational principles,
	it suffices to carefully adjust the flows 
	only at the sites $x = (1,k)$ with $0\le k \le R_1$.
For all other sites, the flows may be modified more coarsely.

In the remainder of this subsection, 
	whenever we write $\Phi(x,y) = s$ for a flow $\Phi$, 
	it implicitly means that $\Phi(y,x) = -s$.

\begin{lemma}\label{lem:flow modification}
	Let $\Phi$ be one of the flows $\phi$ and $\psi$,
		and let $R_2 = \lfloor rN\rfloor$ where $0 < r < \min\{\varepsilon,\frac{\lambda}{1+2\lambda}\}$.
	Then, there exists a flow $\eta$ satisfying
	\[
		\eta(x,y)
		= \begin{cases} \smallskip 
			\frac{1}{Z_{N,\lambda}} C_{N, \lambda} O(N^{-1+2q})
				& \text{if } x = (1,k),\ y = (1,l),\ k\in [0,R_1],\ l\in [0,R_2], \\ \smallskip
			\frac{1}{Z_{N,\lambda}} C_{N, \lambda} O(N^{2+2q})
				& \text{if } x = (1,k),\ y = (1,k+1),\ k \in [R_1+1,N-1], \\ \smallskip
			\frac{1}{Z_{N,\lambda}} \mu(1,k) (1+\lambda)^{-k} O(N)
				& \text{if } x = (1,k),\ y = (1,k+1),\ k \in [n,N-1], \\ \smallskip
			0	& \text{otherwise},
		\end{cases}
	\]
	so that the flow $\widehat{\Phi} = \Phi + \eta$ becomes divergence-free, 
		except at the sites $(0,0)$ and $(1,n)$.
\end{lemma}
\begin{proof}
	By Lemma \ref{lem:div of associated flows}, together with Theorems 
		\ref{thm:stationary asymptotic high} and \ref{thm:stationary asymptotic low}, 
		we have 
	\[
		(\DIV \Phi)(1,k) = \begin{cases} \smallskip
			\frac{1}{Z_{N,\lambda}} C_{N, \lambda} O(N^q)
				& \text{if } 0\le k\le R_1, \\ \smallskip
			\frac{1}{Z_{N,\lambda}} C_{N, \lambda} O(N)
				& \text{otherwise}.
		\end{cases}
	\]

	We first inductively modify the flow 
		to be divergence-free at the sites $(1,k)$ for each $0\le k\le R_1$.
	Set $\Phi_{-1} = \Phi$.
	Fix $k$, and suppose a flow $\Phi_{k-1}$ is divergence-free at all sites $(1,j)$ for $0\le j < k$,
		and that $(\DIV\Phi_{k-1})(1,k) = D_k$.
	Consider the flow $\Phi_k = \Phi_{k-1} + \eta_k$, where
	\[
		\eta_k(x,y)
		= \begin{cases} \smallskip
			(R_2 - k)^{-1} D_k	
				&	\text{if } x = (1,l),\ y = (1,k),\ l \in [k+1,R_2], \\ \smallskip
			0	&	\text{otherwise}.
		\end{cases}
	\]
	Then each term appearing in $\eta_k$ is of order $O(N^{-1})D_k$,
		and the resulting flow $\Phi_k$ is divergence-free at all sites $(1,j)$ for $0\le j\le k$.	
	By repeating this procedure, 
		we obtain a flow $\Phi_{R_1}$ that is divergence-free at every site $(1,k)$ for $0 \le k \le R_1$.
	We readily observe that $\Phi_{R_1}$ has 
		divergence of order $\frac{1}{Z_{N,\lambda}} C_{N, \lambda} O(N^{1+2q})$ at all other sites.
	Define 
	\[
		\widetilde{\eta}(x,y)
		= \begin{cases} \smallskip
			- \sum_{l=0}^{k} (\DIV \Phi_{R_1})(1,l)
				&	\text{if } x = (1,k),\ y = (1,k+1),\ k \in [R_1+1,n-1], \\ \smallskip
			\sum_{l=k+1}^{N} (\DIV \Phi_{R_1})(1,l)
				&	\text{if } x = (1,k),\ y = (1,k+1),\ k \in [n,N-1], \\ \smallskip
			0	&	\text{otherwise}.
		\end{cases}
	\]
	Then the flow $\eta = \sum_{k=1}^{R_1} \eta_k + \widetilde{\eta}$ 
		satisfies the requirements of the lemma.
	The estimates for $\eta((1,k),(1,k+1))$, where $k\in[n,N-1]$, 
		follow from \eqref{eq:div of phi h star} and 
		Theorems \ref{thm:stationary asymptotic high} 
		and \ref{thm:stationary asymptotic low}.
\end{proof}

\begin{proof}[Proof of Theorem \ref{thm:capacity estimate}]
	The proof is straightforward. 
	Let the test functions and flows $(f,\phi)$ and $(g,\psi)$ 
		be defined as in \eqref{eq:test function flow},
		and let $\widehat{\phi}$ and $\widehat{\psi}$ be 
		the modified flows constructed as described in Lemma \ref{lem:flow modification}.
	Applying Theorems \ref{thm:dirichlet} and \ref{thm:thomson}
		with these test functions and flows,
		we establish the following upper and lower bounds for the capacity:
	\[
		\|\Phi_f - \widehat{\phi}\|^2 \quad\text{and}\quad 
		\frac{1}{\|\Phi_g - \widehat{\psi}\|^2} |(\DIV \widehat{\psi})(0,0)|^2,
	\]
	respectively.
	We will show that these two bounds coincide and are equal to 
	\begin{equation}\label{eq:capacity bound coincidence}
		\frac{1}{Z_{N,\lambda}}\lambda N C_{N,\lambda} (1+o(1)),
	\end{equation}
	which matches the right-hand side of \eqref{eq:capacity estimate}
		by Lemma \ref{lem:total mass}.

	Note that 
	\[
		\Phi_f - \phi
		= \psi - \Phi_g
		= \frac{1}{2}(\Phi_h + \Phi_h^*)
		= \Psi_h.
	\]
	By \eqref{eq:u0 asymptotic} and Lemma \ref{lem:sum of scaled vn}, it follows that 
	\begin{align*}
		\|\Psi_h\|^2
		&= \frac{1}{2}\sum_{x,y\in F} c_F^s(x,y)[h(y) - h(x)]^2 \\
		&= \frac{1}{2Z_{N,\lambda}}
			\biggl[ 
				\mu(0,0)\alpha + \sum_{k=0}^{N} \mu(1,k) (1+\lambda)^{-k}
			\biggr] 
		= \frac{1}{Z_{N,\lambda}}\lambda N C_{N,\lambda} (1+o(1)).
	\end{align*}
	Moreover, Lemma \ref{lem:div of associated flows} implies that 
	\[
		(\DIV \psi)(0,0)
		= - \frac{1}{Z_{N,\lambda}}\lambda N C_{N,\lambda} (1+o(1)).
	\]
	Hence, the two terms 
	\[
		\|\Phi_f - \phi\|^2 \quad\text{and}\quad 
		\frac{1}{\|\Phi_g - \psi\|^2} |(\DIV \psi)(0,0)|^2
	\]
	coincide at \eqref{eq:capacity bound coincidence}.

	Let $\eta$ be the modification flow described in Lemma \ref{lem:flow modification}.
	Then $\Psi_h$ is supported on pairs $(x,y)$ where either $x$ or $y$ equals $(0,0)$, 
		while $\eta$ is supported on the complement of these pairs.
	Consequently, $\Psi_h$ and $\eta$ are orthogonal with respect to the flow inner product, 
		and $\eta$ is divergence-free at $(0,0)$.
	
	It remains to show that the norm of $\eta$ is negligible 
		compared to the norm of $\Psi_h$.
	We have 
	\[
		c_F^s((1,k),(1,l)) 
		\gtrsim \begin{cases} \smallskip
			\frac{1}{Z_{N,\lambda}} C_{N,\lambda} 
				& \text{if } k\in [0,R_1],\ l\in [0,R_2], \\ \smallskip
			\frac{1}{Z_{N,\lambda}} C_{N,\lambda} (1+\lambda)^{R_1}
				& \text{if } k\in [R_1+1,W-1],\ l = k+1, \\ \smallskip
			\frac{1}{Z_{N,\lambda}} \mu(1,k) 
			& \text{if } k\in [0,N-1],\ l = k+1, 
		\end{cases}
	\]
	which follows immediately from Theorems \ref{thm:stationary asymptotic high} 
		and \ref{thm:stationary asymptotic low}.
	We divide the quantity $\|\eta\|^2$ into three parts:
	\begin{align*}
		\|\eta\|^2
		\le{}& \sum_{k=0}^{R_1} \sum_{l=1}^{R_2} \frac{1}{c_F^s((1,k),(1,l))}
				\Bigl[ \frac{1}{Z_{N,\lambda}} C_{N, \lambda} O(N^{-1+2q}) \Bigr]^2 \\
			&+ \sum_{k=R_1+1}^{W-1} \frac{1}{c_F^s((1,k),(1,k+1))}
				\Bigl[ \frac{1}{Z_{N,\lambda}} C_{N, \lambda} O(N^{2+2q}) \Bigr]^2 \\
			&+ \sum_{k=W}^{N-1} \frac{1}{c_F^s((1,k),(1,k+1))}
				\Bigl[ \frac{1}{Z_{N,\lambda}} \mu(1,k)(1+\lambda)^{-k} O(N) \Bigr]^2. \\
		={}& \frac{1}{Z_{N,\lambda}} C_{N,\lambda} 
			\Bigl[ O(N^{-1+5q}) + (1+\lambda)^{-R_1} O(N^{5+4q}) \Bigr]
			+ \frac{1}{Z_{N,\lambda}} \mu(1,W) (1+\lambda)^{-2W} O(N^3) \\
		\ll{}& \frac{1}{Z_{N,\lambda}} \lambda N C_{N,\lambda} 
	\end{align*}
	The first two sums can be expressed as 
	\[
		\frac{1}{Z_{N,\lambda}} C_{N,\lambda} 
			\Bigl[ O(N^{-1+5q}) + (1+\lambda)^{-R_1} O(N^{5+4q}) \Bigr],
	\]
	which is negligible compared to $\|\Psi_h\|^2$ when $q > 0$ is sufficiently small.
	We now consider the third sum. 
	If $\lambda > \varphi$, then $W=N$, and there is nothing further to prove. 
	If $\lambda \le \varphi$, then the third sum is bounded above by 
	\[
		\frac{1}{Z_{N,\lambda}} \mu(1,W) (1+\lambda)^{-2W} O(N^3)
		\ll \frac{1}{Z_{N,\lambda}} \lambda N C_{N,\lambda} 
	\]
	due to our definition of $w$, so the proof is complete.
\end{proof}

\subsection{Proof of the Main Theorem}\label{subsection:main theorem proof}
We now prove our main result: the Eyring--Kramers law.

\begin{proof}[Proof of Theorem \ref{thm:eyring kramers}]
	Let $\varepsilon > 0$ be given, and let $x\in \{0,1\}\times [\varepsilon N, N]$. 
	Our goal is to prove the identity 
	\begin{equation}\label{eq:eyring kramers restate}
	\begin{split}
		\EE_{x}\tau_{(0,0)}
		&= \frac{1}{1+2\lambda}
			\Bigl(\frac{1+\lambda}{\lambda}\Bigr)^{\frac{2}{1+2\lambda}}
			B\Bigl(\tfrac{1}{1+2\lambda},N\Bigr)
			\Bigl(\frac{(1+\lambda)^2}{1+2\lambda}\Bigr)^{N}
			(1+o(1)).
	\end{split}
	\end{equation}
	Note that \eqref{eq:eyring kramers restate} is equivalent to \eqref{eq:eyring kramers},
		since Stirling's formula \cite[Equation (5.11.12)]{nist} implies that 
		$B(\frac{1}{1+2\lambda},N) \simeq \Gamma(\frac{1}{1+2\lambda}) N^{-\frac{1}{1+2\lambda}}$.

	We will prove \eqref{eq:eyring kramers restate} in three steps.

	\smallskip

	\noindent\textbf{Step 1.}
	$x=(1,n)$ with $n\in [\varepsilon N, W]$, where $W$ is as defined in Theorem \ref{thm:capacity estimate}.

	\nopagebreak
	By Theorem \ref{thm:capacity estimate},
		we immediately see that the right-hand side of \eqref{eq:eyring kramers restate} equals 
		the inverse of the capacity between the states $(0,0)$ and $(1,n)$.
	Thus, applying the mean hitting time formula from Proposition \ref{prop:mean hit time formula},
		it suffices to show that
	\begin{equation}\label{eq:eyring kramers claim 1}
		\sum_{z\in \mathcal{H}} h_{(1,n),(0,0)}^\dagger(z) \nu(z)
		= 1 + o(1),
	\end{equation}
	where $\mathcal{H} = \{0,1\}\times [0,N]$
		and $\nu$ denotes the quasi-stationary distribution of the process.

	Given that $h_{(1,n),(0,0)}^\dagger(z) \le 1$ for all $z\in \mathcal{H}$,
		our task reduces to establishing a suitable lower bound
		for the left-hand side of \eqref{eq:eyring kramers claim 1}.
	Define $m = \lfloor\frac{\lambda}{1+\lambda}N\rfloor$ and $R = N^{\frac{1}{2}+\varepsilon}$,
		and consider states $z = (1,l)$ with $|l-m| < R$ and $l \neq n$.
	Lemma \ref{lem:total mass} shows that 
		the quasi-stationary distribution of the process is concentrated around such states $z$.
	Furthermore, applying Propositions 
		\ref{prop:time symmetry of capacity},
		\ref{prop:monotonicity of capacity},
		and \ref{prop:equilibrium potential estimate},
		together with \eqref{eq:capacity for trace process},
		we obtain 
	\begin{equation*}
		1 - \frac{\CAP_F((1,l),(0,0))}{\CAP_F((1,l),(1,n))}
		\le h_{(1,n),(0,0)}^\dagger(1,l).
	\end{equation*}
	Therefore, it suffices to verify that 
	\begin{equation}\label{eq:eyring kramers claim 2}
		\CAP_F((1,l),(0,0)) \ll \CAP_F((1,l),(1,n))
	\end{equation}
	holds uniformly in $n$ and $l$ to conclude the assertion.
	
	The left-hand side of \eqref{eq:eyring kramers claim 2} can be estimated
		using Theorem \ref{thm:capacity estimate}.
	Thus, the task reduces to establishing a lower bound for the right-hand side.
	We apply Theorem \ref{thm:thomson} (the Thomson principle) 
		by choosing the test function $g \equiv 0$
		and constructing a test flow $\psi$ satisfying\footnote{
			In this proof, if $l > n$,
			we interpret $[l,n]$ as the interval $[n,l]$.
		}
	\[
		\psi(x,y)
		= \begin{cases}  \smallskip
			\pm 1	& \text{if } x = (1,j),\ y = (1,k),\ j,k\in [l,n],\ |j-k|=1, 	\\  \smallskip
			0		& \text{otherwise}.
		\end{cases}
	\]
	Then $\psi$ is a unit flow from $(1,l)$ to $(1,n)$,
		divergence-free except at the states $(1,l)$ and $(1,n)$.
	Consequently, we obtain 
	\begin{align*}
		\CAP_F((1,l),(1,n))
		\ge \frac{1}{\|\psi\|^2} 
		&= \biggl[
			\sum_{k,k+1\in[l,n]} \frac{1}{c_s((1,k),(1,k+1))}
			\biggr]^{-1} \\
		&\ge \frac{1}{N}\min_{k,k+1\in[l,n]} c_s((1,k),(1,k+1)),
	\end{align*}
	where the summation and minimum are taken over all integers $k$ such that
		both $k$ and $k+1$ belong to the interval $[l,n]$.
	Moreover, it holds that
	\begin{align*}
		c_s((1,k),(1,k+1))
		&= \frac{1}{2 Z_{N,\lambda}} \Bigl[
			\lambda(N-k) \mu(1,k)
			+ \Bigl(k+1+\frac{\lambda}{(1+\lambda)^2}\Bigr)\mu(1,k+1)
		\Bigr] \\
		&\ge \frac{1}{Z_{N,\lambda}}\mu(1,k+1).
	\end{align*}
	Hence, it suffices to show that
	\begin{equation*}
		N^2 C_{N,\lambda} \ll \mu(1,k)
	\end{equation*}
	holds uniformly for all $k\in[\varepsilon N, W]$,
		and this fact is clear from the definition of $W$.

	\smallskip

	\noindent\textbf{Step 2.}
	$x=(0,n)$.

	\nopagebreak
	Suppose that formula \eqref{eq:eyring kramers restate} holds uniformly 
		for initial states $x=(1,n)$ with $n\in[\varepsilon N, M]$,
		where $M = M_N$ is a function of $N$.
	We will show that \eqref{eq:eyring kramers restate} also holds uniformly 
		for $x=(0,n)$ with $n\in[\varepsilon N, M]$.

	Fix $n\in [\varepsilon N, M]$.
	By the monotonicity of contact process,
		the mean extinction time starting from $(0,n)$ is 
		less than or equal to the mean extinction time starting from $(1,n)$.
	This observation establishes one direction of inequality for \eqref{eq:eyring kramers restate}.

	For the opposite inequality,
		note that when the hub is healthy,
		the probability that the subsequent jump of the process results in
		reinfection of the hub is $\frac{\lambda}{1+\lambda}$.
	Hence, the process starting from $(0,n)$ reinfects the hub
		before it reaches the state $(0,\lfloor n/2\rfloor)$ with high probability,
		and thus the desired inequality readily follows.
	
	\smallskip

	\noindent\textbf{Step 3.}
	$x=(1,n)$ with $n\in [m+1,N]$, where $m = \lfloor \frac{\lambda}{1+\lambda}N \rfloor$. 

	\nopagebreak
	Starting from the state $x=(1,n)$, 
		the process must hit one of the states in the set $B=\{(1,m+1),(0,m+1)\}$
		before it can hit the all-healthy state $(0,0)$.
	By Steps 1 and 2, 
		the formula \eqref{eq:eyring kramers restate} holds uniformly for 
		the processes initiated at any state in $B$.
	Thus, to verify the formula for $x=(1,n)$, it remains to show that 
		the mean hitting time $\EE_x\tau_B$ is negligible 
		compared to the right-hand side of \eqref{eq:eyring kramers restate}.
	
	We use a martingale method 
		to derive an upper bound for this mean hitting time.
	Define a function $F$ on the state space $\{0,1\}\times [0,N]$ by 
	\[
		F((o,k))
		= \begin{cases} \smallskip 
			\frac{N-k}{N-m-1}	&	\text{if } k\in[m+2,N], \\ \smallskip 
			1					&	\text{if } k\in[0,m+1],
		\end{cases}
	\]
	where $o\in\{0,1\}$.
	Let $L$ be the generator of the contact process $(x_t)_{t\ge 0} = ((o_t,n_t))_{t\ge 0}$.
	For all $k\in[m+2,N]$, we have 
	\[
		LF(1,k)
		= k \Bigl(\frac{1}{N-m-1}\Bigr) + \lambda(N-k) \Bigl(-\frac{1}{N-m-1}\Bigr)
		\ge \frac{1}{N-m-1}
	\]
	since $k \ge \lambda(N-k) + 1$.
	Similarly, we also have 
	\[
		LF(0,k)
		= k \Bigl(\frac{1}{N-m-1}\Bigr)
		\ge \frac{1}{N-m-1}.
	\]
	Now, consider the martingale $(M_t)_{t\ge 0}$ defined by 
	\[
		M_t 
		= F(x_t) - F(x_0) - \int_0^t LF(x_s)\,ds.
	\]
	By the optional stopping theorem and the above inequalities, 
		we have 
	\[
		0
		= \lim_{t\to\infty} \EE_x M_{t\wedge \tau_B}
		\le 1 - \frac{1}{N-m-1} \EE_x \tau_B.
	\]
	Hence, we conclude that $\EE_x \tau_B = O(N)$, which completes the proof for this step.

	\smallskip

	Combining Steps 1, 2, and 3 covers all initial states in the set $\{0,1\}\times [\varepsilon N, N]$.
\end{proof}

\noindent \textbf{Acknowledgment.} 
This research is supported by the National Research Foundation of Korea (NRF) grant 
funded by the Korea government (MSIT) (No. 2023R1A2C100517311) and 2023 Student-Directed 
Education research program through Faculty of Liberal Education, Seoul National University. 
The author expresses gratitude to Insuk Seo for introducing the problem and offering
	enlightening insights, to Mouad Ramil and Seonwoo Kim for their fruitful discussions
	on the problem, 
	and to the anonymous reviewers for identifying an error in an earlier proof of the main theorem.

\bibliographystyle{authordate1}

\end{document}